\documentclass[a4paper,11pt]{article}
\pdfoutput=1

\usepackage[utf8]{inputenc}
\usepackage[T1]{fontenc}

\usepackage{amsmath,amssymb,amsfonts,amsthm}
\usepackage{mathrsfs}
\usepackage[all]{xy}
\usepackage{hyperref}
\usepackage{geometry}

\title{The complex symplectic geometry of the deformation space of complex projective structures}
\author{Brice Loustau}
\date{}

\newcommand{\Z}{\mathbb{Z}}
\newcommand{\R}{\mathbb{R}}
\newcommand{\C}{\mathbb{C}}
\newcommand{\HH}{\mathbb{H}}

\newcommand{\CPS}{\mathcal{CP}(S)}
\newcommand{\TS}{\mathcal{T}(S)}
\newcommand{\FS}{\mathcal{F}(S)}
\newcommand{\QFS}{\mathcal{QF}(S)}

\newcommand{\RS}{\mathcal{R}(S)}

\newcommand{\XS}{\mathcal{X}(S)}

\newcommand{\slC}{\mathfrak{sl}_2(\C)}

\newcommand{\CP}{\C\mathbf{P}^1}

\DeclareMathOperator{\hol}{\mathit{hol}}
\DeclareMathOperator{\PSL}{\mathit{PSL}}
\DeclareMathOperator{\SO}{\mathit{SO}}
\DeclareMathOperator{\SL}{\mathit{SL}}
\DeclareMathOperator{\pr}{pr}
\DeclareMathOperator{\HB}{\mathit{HB}}

\theoremstyle{plain}
\newtheorem{theorem}{Theorem}[section]
\newtheorem{coro}[theorem]{Corollary}
\newtheorem{propo}[theorem]{Proposition}
\newtheorem{lemma}[theorem]{Lemma}

\theoremstyle{definition}
\newtheorem{defin}[theorem]{Definition}

\newtheorem{remark}[theorem]{Remark}

\begin{document}

\maketitle

This article investigates the complex symplectic geometry of the deformation space of complex projective 
structures on a closed oriented surface of genus at least $2$.
The cotangent symplectic structure given by the Schwarzian parametrization is studied carefully and compared to the 
Goldman symplectic structure on the character variety, clarifying and generalizing a theorem of S. Kawai \cite{kawai}. 
Generalizations of results of C. McMullen are derived,
notably quasi-Fuchsian reciprocity \cite{mcmullenkahler}. 
The symplectic geometry is also described in a Hamiltonian setting with the complex Fenchel-Nielsen coordinates
on quasi-Fuchsian space, recovering results of I. Platis \cite{platis}.

\setcounter{tocdepth}{2}
\pdfbookmark[0]{Contents}{Contents}
\tableofcontents

\section{Introduction}

Complex projective structures on surfaces are rich examples of geometric 
structures. 
They include in particular the three classical homogeneous
Riemannian geometries on surfaces (Euclidean, spherical, hyperbolic) and they extend the theory of complex structures on surfaces, 
{\it i.e.} Teichm\"uller theory. They also have a strong connection to hyperbolic structures on $3$-manifolds. Another feature
is their analytic description using the Schwarzian derivative, which turns the deformation space of complex projective structures
into a holomorphic affine bundle modeled
on the cotangent bundle to Teichm\"uller space. A natural complex symplectic geometry shows through these different perspectives,
 which has been discussed by various authors, e.g. \cite{kawai},
\cite{platis} and \cite{Gsl2}.
This article attempts a unifying picture of the complex symplectic geometry
of the deformation space of complex projective structures on surfaces, one that carefully relates the different approaches
\footnote{Of course, this has already been done at least partially by authors including the three previously mentioned.}.

\begin{center}
\rule{120pt}{.5pt}
\end{center}
\vspace{1em}

Let $S$ be a closed oriented surface of genus $g\geqslant 2$. A complex projective structure on $S$ is given
by an atlas of charts mapping open sets of $S$ into the projective line $\CP$ such that the transition maps are restrictions
of projective linear transformations. The deformation space of projective structures $\CPS$ is the space of equivalence classes
of projective structures on $S$, where two projective structures are considered equivalent
if they are diffeomorphic\footnote{More precisely, diffeomorphic by a homotopically trivial diffeomorphism, see section \ref{defts}.}.
Any projective atlas is in particular a holomorphic atlas, therefore a projective structure defines an underlying complex structure. 
This gives
a forgetful projection $p : \CPS \to \TS$, where $\TS$ is the Teichm\"uller space of $S$, defined as the deformation 
space of complex structures on $S$.

The Schwarzian derivative is a differential operator that turns the fibers of $p$ into complex affine spaces.
Globally, $\CPS$ is a holomorphic affine bundle modeled on the holomorphic cotangent bundle $T^*\TS$. 
This yields an identification
$\CPS \approx T^*\TS$, but it is not canonical: it depends on the choice of the ``zero section'' $\sigma : \TS \to \CPS$. 
There
are at least two natural choices of sections to be considered. 
The Fuchsian section $\sigma_{\cal F}$ assigns to a Riemann surface $X$
its Fuchsian projective structure given by the uniformization theorem. However, $\sigma_{\cal F}$ is not holomorphic. 
The other natural choice
is that of a Bers section, given by Bers' simultaneous uniformization theorem. 
Bers sections are a family of holomorphic sections
parametrized by Teichm\"uller space. Like any holomorphic cotangent bundle, $T^*\TS$ is equipped with a canonical complex symplectic 
form $\omega_{\mathrm{can}}$.
Each choice of a zero section $\sigma$ thus yields a symplectic structure $\omega^\sigma$ on $\CPS$, 
simply by pulling back the canonical symplectic form of
$T^*\TS$. A first natural question is: How is $\omega^\sigma$ affected by $\sigma$? 
A small computation shows:
\newtheorem*{craca}{Proposition \ref{craca}}
\begin{craca}
 For any two sections $\sigma_1$ and $\sigma_2$ to $p: \CPS \to \TS$,
\begin{equation}\omega^{\sigma_2} - \omega^{\sigma_1} = - p^* d(\sigma_2 - \sigma_1)~.\end{equation}
\end{craca}

A significantly different description of $\CPS$ is given by the holonomy of complex projective structures. Holonomy is a concept defined
for any geometric structures, in this situation it gives a local identification 
$\hol : \CPS \to {\cal X}(S,\PSL_2(\C))$, where the character
variety ${\cal X}(S, \PSL_2(\C))$ is defined as a quotient of the set of representations 
$\rho : \pi_1(S) \to \PSL_2(\C)$. By a general construction of
Goldman, ${\cal X}(S, \PSL_2(\C))$ enjoys a natural complex symplectic structure $\omega_G$. 
Does this symplectic structure compare to the cotangent symplectic
structures $\omega^\sigma$ introduced above? A theorem of Kawai \cite{kawai} gives a pleasant answer to that question: If $\sigma$ is any Bers
section, then $\omega^{\sigma}$ and $\omega_G$\footnote{We mean here $\hol^*\omega_G$ rather than $\omega_G$, but we abusively use the same notation for the two
(as explained in section \ref{holonomy}).} agree up to some constant. Kawai's proof is highly technical and not very insightful though.
Also, the conventions chosen in his paper can be 
misleading\footnote{With the conventions chosen in his paper, Kawai finds $\omega^{\sigma} = \pi \omega_G$. Compare with our result: 
$\omega^{\sigma} = -i \omega_G$. Whether the constant is real or imaginary does matter when taking the real and imaginary parts, obviously,
and this can be significant. Kawai's choices imply that
$\omega_G$ takes imaginary values in restriction to the Fuchsian slice, which does not seem very relevant. Goldman showed in \cite{goldmannature}
that (with appropriate conventions) $\omega_G$ is just the Weil-Petersson K\"ahler form on the Fuchsian slice. For the interested reader, we believe that, 
even after rectifying the conventions, there is a factor $2$ missing in Kawai's result.}. Relying on theorems of other authors,
we give a simple alternative proof of Kawai's result. In fact, we are able to do a little better and completely answer the question raised above. Our argument 
is based on the observation that there is an intricate circle of related ideas:
\begin{itemize}
 \item[$\mathbf{(i)}$]  $p : \CPS \to \TS$ is a Lagrangian fibration (with respect to $\omega_G$).
 \item[$\mathbf{(ii)}$] Bers sections $\TS \to \CPS$ are Lagrangian (with respect to $\omega_G$).
 \item[$\mathbf{(iii)}$] If $M$ is a $3$-manifold diffeomorphic to $S \times \R$, 
 then the Bers simultaneous uniformization map $\beta : {\cal T}(\partial_\infty M) \to {\cal CP}(\partial_\infty M)$
 is Lagrangian (with respect to $\omega_G$).
 \item[$\mathbf{(iv)}$] $\omega_G$ restricts to the Weil-Petersson K\"ahler form $\omega_{WP}$ on the Fuchsian slice.
 \item[$\mathbf{(v)}$] If $\sigma$ is any Bers section, then $d(\sigma_{\cal F}-\sigma) = -i \omega_{WP}$.
 \item[$\mathbf{(vi)}$] McMullen's quasi-Fuchsian reciprocity (see \cite{mcmullenkahler} and Theorem \ref{qfrcp}).
 \item[$\mathbf{(vii)}$] For any Bers section $\sigma$, $\omega^{\sigma} = -i \omega_G$.
\end{itemize}
Let us briefly comment on these. $\mathbf{(iv)}$ is a result due to Goldman (\cite{goldmannature}). $\mathbf{(v)}$ and $\mathbf{(vi)}$ are closely
related and due to McMullen (\cite{mcmullenkahler}). Steven Kerckhoff discovered that $\mathbf{(iii)}$ 
easily follows from a standard argument, we include this argument in our presentation (Theorem \ref{lagemb}) for
completeness. $\mathbf{(vii)}$ appears to be the strongest result, as it is not too hard to see that it implies all other results\footnote{
This is not entirely true {\it per se}, but we do not want to go into too much detail here.}. However, using Proposition \ref{craca}
written above and a simple analytic continuation argument (Theorem \ref{bilibili}), we show that $\mathbf{(iv)}$ and $\mathbf{(v)}$ imply $\mathbf{(vii)}$.
In fact, we give a characterization
of sections $\sigma$ such that $\omega^\sigma$ agrees with $\omega$:
\newtheorem*{malauvent}{Theorem \ref{malauvent}}
\begin{malauvent}
Let $\sigma : \TS \rightarrow \CPS$ be a section to $p$. Then $\omega^\sigma$ agrees with the standard complex symplectic structure $\omega_G$ on $\CPS$
if and only if $\sigma_{\cal F} - \sigma$ is a primitive for the Weil-Petersson metric on $\TS$:
\begin{equation}\omega^{\sigma} = \omega_G ~ \Leftrightarrow ~ d(\sigma_{\cal F} - \sigma) = \omega_{WP}~.\end{equation}
\end{malauvent}
$\mathbf{(vii)}$ then follows from McMullen's theorem $\mathbf{(v)}$:
\newtheorem*{main1}{Theorem \ref{main1}}
\begin{main1}
If $\sigma : \TS \rightarrow \CPS$ is any Bers section, then \begin{equation}\omega^\sigma = -i \omega_G~.\end{equation}
\end{main1}
We also get the expression of the symplectic structure pulled back by the Fuchsian identification:
\newtheorem*{fghh}{Corollary \ref{fghh}}
\begin{fghh}
Let $\sigma_{\cal F} : \TS \to \CPS$ be the Fuchsian section. Then
\begin{equation}\omega^{\sigma_{\cal F}} = -i(\omega_G - p^* \omega_{WP})~.\end{equation}
\end{fghh}
Generalizing these ideas in the setting of convex cocompact $3$-manifolds, we prove a generalized version of Theorem \ref{main1}, 
relying on a result of 
{Takhtajan}-{Teo}\cite{takteo}\footnote{However, we stress that the proof also relies indirectly on Theorem \ref{main1}.}:
\newtheorem*{main2}{Theorem \ref{main2}}
\begin{main2}
Let $\sigma : \TS \to \CPS$ be a generalized Bers section. Then \begin{equation}\omega^\sigma = -i \omega_G~.\end{equation}
\end{main2}
We derive a generalization of McMullen's result $\mathbf{(v)}$:
\newtheorem*{coui}{Corollary \ref{coui}}
\begin{coui}
Let $\sigma : \TS \to \CPS$ be a generalized Bers section. Then \begin{equation}d(\sigma_{\cal F} - \sigma) = -i\omega_{WP}~.\end{equation}
\end{coui}
and a generalized version of McMullen's quasi-Fuchsian reciprocity:
\newtheorem*{qfrcp}{Theorem \ref{qfrcp}}
\begin{qfrcp}
Let $f : {\cal T}(S_j) \to {\cal CP}(S_k)$
and $g : {\cal T}(S_k) \to {\cal CP}(S_j)$ be reciprocal generalized Bers embeddings. Then $D_{X_j}f$ and $D_{X_k}g$ are dual maps.
In other words, for any $\mu\in T_{X_j}{\cal T}(S_j)$ and $\nu\in T_{X_k}{\cal T}(S_k)$,
\begin{equation} \langle D_{X_j}f(\mu),\nu \rangle = \langle \mu,D_{X_k}g(\nu) \rangle~.\end{equation}
\end{qfrcp}

Finally, we discuss the symplectic geometry of $\CPS$ in relation
to complex Fenchel-Nielsen coordinates on the quasi-Fuchsian space $\QFS$. These are global holomorphic coordinates on $\QFS$ introduced by
Kourouniotis (\cite{K3}) and Tan (\cite{tan}) that are the complexification of the classical Fenchel-Nielsen
coordinates on Teichm\"uller space $\TS$, or rather the Fuchsian space $\FS$. In \cite{W1}, \cite{W2} and \cite{W3}, Wolpert showed
that the Fenchel-Nielsen coordinates on $\FS$ encode the symplectic structure. For any simple closed
curve $\gamma$ on the surface $S$, there is a hyperbolic length function $l_\gamma : \FS \to \R$ and a twist flow $\mathrm{tw}_\gamma : \R \times \FS \to \FS$.
Given a pants decomposition $\alpha = (\alpha_1,\dots,\alpha_N)$\footnote{{\it i.e.} a maximal collection
of nontrivial distinct free homotopy classes of simple closed curves, see section \ref{reio}.} on $S$, 
choosing a section to $l_\alpha = (l_{\alpha_1}, \dots, l_{\alpha_N})$ yields the classical Fenchel-Nielsen coordinates on Teichm\"uller space
$(l_\alpha,\tau_\alpha) :  \FS \to {(\R_{>0})}^N \times \R^N$. Wolpert showed that the twist flow associated to a curve $\gamma$ is the Hamiltonian flow
of the length function $l_\gamma$. He also gave formulas for the Poisson bracket of two length functions, which show in particular that the length functions
$l_{\alpha_i}$ associated to a pants decomposition $\alpha$ define an integrable Hamiltonian system, for which the functions $l_{\alpha_i}$ are the action
variables and the twist functions $\tau_{\alpha_i}$ are the angle variables. In \cite{platis}, Platis shows that this very nice ``Hamiltonian picture'' remains
true in its complexified version on the quasi-Fuchsian space for some complex symplectic structure $\omega_P$, giving complex versions of Wolpert's results.
This Hamiltonian picture is also extensively explored on the $\SL_2(\C)$-character variety by Goldman in \cite{Gsl2}. Independently from Platis' work,
our analytic continuation argument shows that complex Fenchel-Nielsen coordinates are Darboux coordinates
for the symplectic structure on $\QFS$:
\newtheorem*{main3}{Theorem \ref{main3}}
\begin{main3}
Let $\alpha$ be any pants decomposition of $S$. Complex Fenchel-Nielsen coordinates $(l_\alpha^\C,\beta_\alpha^\C)$ on the 
quasi-Fuchsian space $\QFS$ are
Darboux coordinates for the standard complex symplectic structure:
\begin{equation}\omega_G = \sum_{i=1}^N dl_{\alpha_i}^\C \wedge d\tau_{\alpha_i}^\C\end{equation}
\end{main3}
and this shows in particular that
\newtheorem*{za}{Corollary \ref{za}}
\begin{za}
Platis' symplectic structure $\omega_P$ is equal to the standard complex symplectic structure $\omega_G$ on the
quasi-Fuchsian space $\QFS$\footnote{This fact is mentioned as ``apparent'' in \cite{platis} and is implied in \cite{Gsl2},
but it would seem that it was not ``formally proved''.}.
\end{za}
We thus recover Platis' and some of Goldman's results, in particular that the complex twist flow is the Hamiltonian flow of the associated complex
length function. Although in the Fuchsian case it would seem unnecessarily sophisticated to use this as a definition of the twist flow,
this approach might be fruitful in the space of projective structures. This transformation relates
to what other authors have called \emph{quakebends} or \emph{complex earthquakes} discussed by Epstein-Marden \cite{EP}, McMullen \cite{mcmullenearthquakes}, 
Series \cite{series} among others.

\begin{center}
\rule{120pt}{.5pt}
\end{center}
\vspace{1em}

\emph{Note:} The study of Taubes' symplectic structure on the deformation space of \emph{minimal
hyperbolic germs}\footnote{We refer to \cite{taubes} and \cite{article2} for the interested reader.}
(in restriction to \emph{almost-Fuchsian space}, which embeds as an open subspace of quasi-Fuchsian space
$\QFS$) is addressed in a forthcoming paper (\cite{article2}).
Both articles are based on the author's PhD thesis (\cite{bricethesis}).

\bigskip

\emph{Structure of the paper:} Section \ref{TSCPS} reviews complex projective structures, 
Fuchsian and quasi-Fuchsian projective structures, the relation between complex projective structures and hyperbolic $3$-manifolds,
(generalized) Bers sections and embeddings. Section \ref{cotangentst} introduces
the affine cotangent symplectic structures given by the Schwarzian parametrization of $\CPS$. Section \ref{charGs} reviews
the character variety, holonomy of projective structures, Goldman's symplectic structure and some of its properties.
In section \ref{complexFNP}, we briefly review Wolpert's ``Hamiltonian picture'' of Teichmüller space, then describe
the complex Fenchel-Nielsen coordinates in quasi-Fuchsian space and Platis' symplectic structure. In section \ref{compsst},
we describe an analytic continuation argument then discuss and compare the different symplectic structures previously 
introduced. Our results are essentially contained in that last section.

\bigskip

\emph{Acknowledgments:} This paper is based on part of the author's PhD thesis, which was supervised by Jean-Marc Schlenker. I wish
to express my gratitude to Jean-Marc for his kind advice. I would also like to thank Steven Kerckhoff, Francis Bonahon, Bill Goldman, 
David Dumas, Jonah Gaster, Andy Sanders, Cyril Lecuire, Julien Marché, among others with whom I have had helpful discussions.

\bigskip

The research leading to these results has received funding from the European Research Council under the {\em European Community}'s 
seventh Framework Programme (FP7/2007-2013)/ERC {\em  grant agreement}.

\section{Teichm\"uller space and the deformation space of complex projective structures}\label{TSCPS}

\subsection{\texorpdfstring{$\TS$ and $\CPS$}{{T(S)} and {CP(S)}}}\label{defts}

Let $S$ be a surface. Unless otherwise stated, we will assume that $S$ is connected\footnote{In some 
sections (e.g. \ref{cphyp}), we will allow $S$ to be disconnected in order to be able to consider the case where $S$
is the boundary
of a compact $3$-manifold. This does not cause any issue in the exposition above.}, oriented, smooth, 
closed and with genus $g \geqslant 2$.

A complex structure on $S$ is a maximal atlas of charts mapping open sets of $S$ into the complex line $\C$ such that
the transition maps are holomorphic.
The atlas is required to be compatible with the orientation and smooth structure on $S$.
A Riemann surface $X$ is a surface $S$ equipped with a complex structure.

The group $\mathrm{Diff}^+(S)$ of orientation preserving diffeomorphisms of $S$ acts on the set
of all complex structures on $S$ in a natural way:
a compatible complex atlas on $S$ is pulled back to another one by such diffeomorphisms.
Denote by $\mathrm{Diff}^+_0(S)$ the identity component of $\mathrm{Diff}^+(S)$, its elements are the 
orientation preserving diffeomorphisms of $S$ that are homotopic to the identity.
The quotient $\TS$ of the set of all complex structures on $S$ by $\mathrm{Diff}^+_0(S)$
is called the \emph{Teichm\"uller space} of $S$, its elements are called \emph{marked Riemann surfaces}.

In a similar fashion, define a \emph{complex projective structure} on $S$ as a maximal atlas of charts 
mapping open sets of $S$ into the complex 
projective line $\CP$ such that the transition maps are (restrictions of) projective linear transformations 
(\textit{i.e.} Möbius transformations of the Riemann sphere). The atlas is also required to be compatible 
with the orientation and smooth structure on $S$. A complex projective surface $Z$ is a surface $S$ equipped with a complex 
projective structure.
In terms of geometric structures (see e.g. \cite{thurston}),
a complex projective structure is a $(\CP,\PSL_2(\C))$-structure.

Again, $\mathrm{Diff}^+(S)$ naturally acts on the set of all complex projective structures on $S$. The quotient $\CPS$ by the subgroup
$\mathrm{Diff}^+_0(S)$ is called the \emph{deformation space of complex projective structures} on $S$, its
elements are \emph{marked complex projective surfaces}.

\subsubsection*{\texorpdfstring{$\TS$ and $\CPS$ are complex manifolds}{{T(S)} and {CP(S)} are complex manifolds}}\label{pif}

Kodaira-Spencer deformation theory (see \cite{KD}, also \cite{earleeells}) applies and it shows that 
$\TS$ is naturally a complex manifold with holomorphic
tangent space given by $T^{1,0}_X \TS = \check{H}^1(X,\Theta_X)$,
where $\Theta_X$ is the sheaf of holomorphic vector fields on $X$. 
Denote by $K$ the canonical bundle over $X$ (the holomorphic cotangent bundle of $X$).
By Dolbeault's theorem, $\check{H}^1(X,\Theta_X)$ is isomorphic to the Dolbeault cohomology space $H^{-1,1}(X)$.
Elements of $H^{-1,1}(X)$ are Dolbeault classes of smooth sections of $K^{-1} \otimes \bar{K}$, 
which are called \emph{Beltrami differentials}.
In a complex chart $z : U \subset S \rightarrow \C$, a Beltrami differential $\mu$
has an expression of the form $\mu = u(z) \frac{d\overline{z}}{dz}$ where $u$ is a smooth function.
The fact that we only consider (Dolbeault) classes of Beltrami differentials can be expressed as follows : 
if $V$ is a vector field on $X$ of type $(1,0)$, then the Beltrami differential $\overline{\partial} V$ induces
a trivial (infinitesimal) deformation of the complex structure $X$.
Recall that $X$ carries a unique hyperbolic metric within its conformal class (called the Poincar\'e metric) 
by the uniformization theorem.
By Hodge theory, every Dolbeault cohomology class has a unique harmonic representative $\mu$.
The tangent space $T_X \TS$ is thus also identified with the space $\HB(X)$ of harmonic Beltrami differentials.

We can also derive a nice description of the Teichm\"uller cotangent space using cohomology machinery.
$\check{H}^1(X,\Theta_X) = H^1(X,K^{-1})$ because $\mathrm{dim}_\C X= 1$ and this space is dual to $H^0(X,K^{2})$ by Serre duality.
An element $\varphi \in Q(X) := H^0(X,K^{2})$ \label{qx} is called a \emph{holomorphic quadratic differential}.
In a complex chart $z : U \subset S \rightarrow \C$, $\varphi$
has an expression of the form $\varphi = \phi(z) dz^{2}$, where $\phi$ is a holomorphic function.
The holomorphic cotangent space ${T}^*_X \TS$ is thus identified with the space $Q(X)$ of holomorphic quadratic differentials.
The duality pairing $Q(X) \times H^{-1,1}(X) \to \C$ is just given by $(\varphi,\mu) \mapsto \int_S \varphi \cdot \mu$.
Note that we systematically use tensor contraction (when dealing with line bundles over $X$) : $\varphi \cdot \mu$ 
is a section of $K \otimes \bar{K} \approx |K|^2$, so it defines a conformal density and can be integrated over $S$.
With the notations above, $\varphi \cdot \mu$ has local expression $\phi(z) u(z) |dz|^2$.

An easy consequence of the Riemann-Roch theorem is that $\mathrm{dim}_\C Q(X) = 3g-3$, so that $\TS$ is a complex manifold 
of dimension $3g-3$.

Similarly, Kodaira-Spencer deformation theory applies to show that $\CPS$ is naturally a complex manifold with
tangent space $T_Z \CPS = \check{H}^1(Z,\Xi_Z)$, where $\Xi_Z$ is the sheaf of projective vector fields on $Z$ (see also \cite{hubbard}).
It follows that $\CPS$ is a complex manifold of dimension $6g-6$.

Unlike Teichm\"uller tangent vectors, there is no immediate way to describe tangent vectors to $\CPS$ in a more tangible way.
However, note that a complex projective atlas is in particular a holomorphic atlas,
so that a complex projective surface $Z$ has an underlying structure of a 
Riemann surface $X$. This yields a forgetful map
\begin{equation} p : \CPS \to \TS \label{pfor} \end{equation} which is easily seen to be holomorphic.
We will see in section \ref{affinebundle} that the fiber $p^{-1}(X)$ is naturally a complex affine space whose underlying
vector space is $Q(X)$. In particular $\dim_\C \CPS = \dim_\C \TS \times \dim_\C Q(X) = 6g-6$ as expected.

\subsubsection*{\texorpdfstring{The Weil-Petersson K\"ahler metric on $\TS$}{The Weil-Petersson K\"ahler metric on {T(S)}}}

The \emph{Weil-Petersson product} of two holomorphic quadratic differentials $\Phi$ and $\Psi$ is given by
\begin{equation}
 \left<\varphi,\psi\right>_{WP} = -\frac{1}{4} \int_X \varphi \cdot {\sigma}^{-1} \cdot \overline{\psi} 
\end{equation}
where ${\sigma}^{-1}$ is the dual current of the area form $\sigma$ for the Poincar\'e metric\footnote{In
a complex chart with values in the upper half-plane $z = x + i y : U \subset X \to \HH^{2}$, 
the tensor product $-\frac{1}{4} \varphi \cdot {\sigma}^{-1} \cdot \overline{\psi}$ reduces to the classical expression
$\begin{displaystyle}
y^2 \varphi(z) \overline{\psi(z)} dx \wedge dy
\end{displaystyle}$.}. It is a Hermitian 
inner product on the complex vector space $Q(X)$.

By duality, this gives a Hermitian product also denoted by $\left<\cdot,\cdot\right>_{WP}$ on $H^{-1,1}(X)$ and globally a Hermitian
metric $\left<\cdot,\cdot\right>_{WP}$ on the manifold $\TS$.
It was first shown to be K\"ahler by Ahlfors \cite{ahlfors2}
and Weil.

The K\"ahler form of the Weil-Petersson metric on $\TS$ is the real symplectic form
\begin{equation}
 \omega_{WP} = - \mathrm{Im} \left<\cdot,\cdot\right>_{WP}~. \label{WP}
\end{equation}

\subsection{Fuchsian and quasi-Fuchsian projective structures}\label{fqf}

Note that whenever a Kleinian group $\Gamma$ ({\it i.e.} a discrete subgroup of $\PSL_2(\C)$) acts freely and properly 
on some open subset $U$
of the complex projective line $\CP$, the quotient surface $U/\Gamma$ inherits a complex projective structure.
This gives a variety (but not all) of complex projective surfaces, called embedded projective structures.

Fuchsian projective structures are a fundamental example of embedded projective structures. 
Given a marked complex structure $X$, the uniformization theorem provides a representation $\rho : \pi_1(S) \to \PSL_2(\R)$
such that $X$ is isomorphic to $\HH^2/\rho(\pi_1(S))$ as a Riemann surface (where $\HH^2$ is the upper half-plane).
$\HH^2$ can be seen as an open set (a disk) in $\CP$ and the Fuchsian group $\rho(\pi_1(S)) \subset \PSL_2(\R)$ is in particular a 
Kleinian group, so the quotient
$X \approx \HH^2/\rho(\pi_1(S))$ inherits a complex projective structure $Z$. This defines a section
\begin{equation}\label{fufuch}
 \sigma_{\cal F} : \TS \to \CPS
\end{equation}
to $p$, called the \emph{Fuchsian section}.
It shows in particular that the projection $p$ is surjective. 
We call $\FS:=\sigma_{\cal F}(\TS)$ the (deformation) space of (standard) Fuchsian (projective) structures on $S$, 
it is an embedded copy of $\TS$ in $\CPS$.

\emph{Quasi-Fuchsian structures} are another important class of embedded projective structures.
Given two marked complex structures $(X^+,X^-) \in \TS \times {\cal T}(\overline{S})$\footnote{
Note that ${\cal T}(\overline{S})$ is canonically identified with $\overline{\TS}$, 
which denotes the manifold $\TS$ equipped with the opposite complex structure.
The same remark holds for ${\cal CP}(\overline{S})$ and $\overline{\CPS}$.}
 (where $\overline{S}$ is the surface $S$ with reversed orientation),
Bers' simultaneous uniformization theorem states that there exists a
unique representation $\rho : \pi_1(S) \stackrel{\sim}{\to} \Gamma \subset \PSL_2(\C)$ up to conjugation such that:
\begin{itemize}
 \item[$\bullet$] The limit set\footnote{The \emph{limit set} $\Lambda = \Lambda(\Gamma)$ 
 is defined as the complement in $\CP$ of the domain
of discontinuity $\Omega$, which is the maximal open set on which $\Gamma$ acts freely and properly. Alternatively,
$\Lambda$ is described as the closure in $\CP$ of the set of fixed points of elements of $\Gamma$.} $\Lambda$ is a Jordan curve.
 The domain of discontinuity $\Omega$ is then 
 the disjoint union of two simply connected domains $\Omega^+$ and $\Omega^-$. A such $\Gamma$ is called a
 quasi-Fuchsian group.
 \item[$\bullet$] As marked Riemann surfaces, $X^+ \approx \Omega^+/\Gamma$ and $X^- \approx \Omega^-/\Gamma$.
\end{itemize}
Again, both Riemann surfaces $X^+$ and $X^-$ inherit embedded complex projective structures $Z^+$ and $Z^-$ by this construction.
This defines a map
$\beta  = (\beta^+,\beta^-) : 
\TS \times  {\cal T}(\overline{S}) \to \CPS \times  {\cal CP}(\overline{S})$
which is a holomorphic section to $p \times p : \CPS \times  {\cal CP}(\overline{S}) \to  \TS \times  {\cal T}(\overline{S})$
by Bers' theorem.
The map $\beta$ has the obvious symmetry property: $\overline{\beta^-(X^+,X^-)} = \beta^+(\overline{X^-},\overline{X^+})$.

In particular, when $X^- \in {\cal T}(\overline{S})$ is fixed, the map 
$\sigma_{X^-} := \beta^+(\cdot,X^-) : \TS \to \CPS \label{bbss}$
is a holomorphic section to $p$, called a \emph{Bers section}, 
and its image $\sigma_{X^-}(\TS)$ in $\CPS$ is called a \emph{Bers slice}. \label{berssection}
On the other hand, when $X^+ \in\TS$ is fixed, the map 
$f_{X^+} = \beta^+(X^+,\cdot) $
 is an embedding of ${\cal T}(\overline{S})$
 in the fiber $P(X^+) := p^{-1}(X^+) \subset \CPS$\footnote{which has the structure of a complex affine space as we 
 will see in section \ref{schwarzian}.},
 $f_{X^+}$ is called a \emph{Bers embedding}.
Also, note that $\sigma_{\cal F}(X) = \beta^+(X,\bar{X}) = \overline{\beta^-(X,\bar{X})} = \sigma_{\bar{X}}(X)$. This shows that the Fuchsian section
$\sigma_{\cal F}$ is real analytic but not holomorphic, in fact it is a maximal totally real analytic embedding, see section \ref{ancont}.

$\QFS:=\beta^+(\TS \times  {\cal T}(\overline{S})) \subset \CPS$ \label{qfdef} is called the (deformation) space of
(standard) quasi-Fuchsian (projective) structures on $S$. It is an open neighborhood of $\FS$ in $\CPS$
(this is a consequence of general arguments mentioned in the next paragraph), 
and it follows from the discussion above that Bers slices and Bers embeddings
define two transverse foliations of $\QFS$ by holomorphic copies of $\TS$.

\subsection{\texorpdfstring{Complex projective structures and hyperbolic $3$-manifolds}{Complex projective structures and hyperbolic 3-manifolds}} \label{cphyp}

In this paragraph, we briefly review the relation between complex projective structures on the boundary of a compact $3$-manifold $\hat{M}$ 
and hyperbolic structures
on its interior. The quasi-Fuchsian projective structures presented in the 
previous section occur as a particular case of this discussion. 
We then define \emph{generalized Bers sections} and \emph{generalized Bers embeddings}, and
fix a few notations for later sections.

Let $M$ be a connected complete hyperbolic $3$-manifold. The universal cover of $M$ is isometric to hyperbolic $3$-space $\HH^3$, 
this defines a unique faithful representation $\rho : \pi_1(M) \to \mathrm{Isom}^+(\HH^3) \approx \PSL_2(\C)$ up to conjugation such that 
$\Gamma:=\rho(\pi_1(M))$
acts freely and properly on $\HH^3$ and $M \approx \HH^3/\Gamma$. Let $\Omega \subset \CP$ be the domain of discontinuity of the Kleinian group $\Gamma$, it is 
the maximal open set on which $\Gamma$ acts freely and properly.
Here $\CP$ is seen as the ``ideal boundary'' of $\HH^3$, also denoted $\partial_\infty \HH^3$.
The possibly disconnected surface $\partial_\infty M := \Omega / \Gamma$ is called the \emph{ideal boundary} 
\label{idboundef} of $M$ and it inherits an embedded complex 
projective structure as the quotient of $\Omega \subset \CP$ by the Kleinian group $\Gamma$.
Conversely, any torsion-free Kleinian group $\Gamma$ acts freely and properly on $\HH^3 \sqcup \Omega$ 
(where $\Omega$ is the domain of discontinuity
of $\Gamma$), and the quotient consists of a $3$-manifold 
$\hat{M} = M \sqcup \partial_\infty M$, where $M = \HH^3 / \Gamma$ is a complete hyperbolic $3$-manifold 
and $\partial_\infty M = \Omega/\Gamma$ is its ideal boundary. In general, the manifold $\hat{M}$ is not compact, if it is then $\hat{M}$
is topologically the end compactification of $M$. In that case we say that the hyperbolic structure on $M$ is \emph{convex cocompact}.
The \emph{convex core} of $M$ is the quotient of the convex hull of the limit set $\Lambda$
in $\HH^3$ by $\Gamma$. 
It is well-known that $M$ is convex cocompact if and only if its convex core is a compact deformation retract of $M$.

Consider now a smooth $3$-manifold with boundary $\hat{M}$ with the following topological restrictions: 
$\hat{M}$ is connected, oriented, compact, irreducible\footnote{meaning that every embedded $2$-sphere
bounds a ball.}, atoroidal\footnote{meaning that it does not contain any embedded, non-boundary parallel, incompressible tori.}
and with infinite fundamental group.
Let $M = \hat{M} \setminus \partial \hat{M}$
denote the interior of $\hat{M}$. For simplicity, we also assume that the boundary $\partial \hat{M}$ is incompressible\footnote{
meaning that the map $\iota_* : \pi_1(\partial \hat{M}) \to \pi_1(\hat{M})$ induced by the inclusion map $\iota$ is injective.} and contains no tori,
so that it consists of a finite number of surfaces $S_1, \dots, S_N$ of genera at least $2$.
The Teichm\"uller space ${\cal T}(\partial \hat{M})$ is described as the direct product 
${\cal T}(\partial \hat{M}) = {\cal T}(S_1) \times \dots \times {\cal T}(S_N)$,
similarly ${\cal CP}(\partial \hat{M}) = {\cal CP}(S_1) \times \dots \times {\cal CP}(S_N)$
and there is a holomorphic ``forgetful'' projection
$p = p_1 \times \dots \times p_N : {\cal CP}(\partial \hat{M}) \rightarrow {\cal T}(\partial \hat{M})$.
Let  $\pr_k : {\cal CP}(\partial \hat{M}) \rightarrow {\cal CP}(S_k)$ denote the 
$k^{\textrm{th}}$ \label{defprmap} projection map. Let us consider the space ${\cal HC}(M)$ of convex cocompact
hyperbolic structures on $M$ up to homotopy. In other words, we define ${\cal HC}(M)$ as the quotient of the set of convex cocompact hyperbolic metrics on $M$
by the group of orientation-preserving diffeomorphisms of $M$ that are homotopic to the identity. Let us mention that Marden \cite{mardeng}
and Sullivan \cite{sullivang} showed that
${\cal HC}(M)$ is a connected component of the interior of the subset of discrete and faithful representations in the 
character variety ${\cal X}(M,PSL_2(\C))$. By the discussion above, any element of ${\cal HC}(M)$ determines a marked
complex projective structure $Z \in {\cal CP}(\partial \hat{M})$. We thus have a map $\varphi : {\cal HC}(M) \to {\cal CP}(\partial \hat{M})$, and it is shown to be 
holomorphic, this is a straightforward consequence of the fact 
that the holonomy map is holomorphic (see section
\ref{holonomy}). Considering the induced conformal structure on $\partial \hat{M}$, define the map $\psi = p \circ \varphi$ as in the following diagram:
$$\xymatrix{{\cal HC}(M) \ar[r]^\varphi \ar[rd]^{\psi} & {\cal CP}(\partial \hat{M}) \ar[d]^p \\
& {\cal T}(\partial \hat{M})~.}$$
The powerful theorem mainly due to Ahlfors, Bers, Kra, Marden, Maskit, Sullivan and Thurson\footnote{see \cite{CanaryMC} chapter 7. for a detailed exposition
of this theorem,
containing in particular the description of the different contributions of the several authors.
A non-exhaustive list of references includes \cite{abg}, 
\cite{ahlforsg}, \cite{bersg}, \cite{krag}, \cite{mardeng}, \cite{maskitg}, \cite{sullivang}.}
says in this context that:
\begin{theorem}[Ahlfors, Bers, Kra, Marden, Maskit, Sullivan, Thurston]
The map $\psi : {\cal HC}(M) \rightarrow {\cal T}(\partial \hat{M})$ is bijective.
\end{theorem}
Let us mention that this statement has to be slightly modified if $\hat{M}$ has compressible boundary.
As a consequence of this theorem, we get
\begin{propo}
The map \begin{equation}\beta = \varphi \circ \psi^{-1} : {\cal T}(\partial \hat{M}) \to {\cal CP}(\partial \hat{M})\end{equation} is a
holomorphic section to $p : {\cal CP}(\partial \hat{M}) \to {\cal T}(\partial \hat{M})$.
\end{propo}
We call $\beta$ the \emph{(generalized) simultaneous uniformization section}. \label{betasimul}
This map allows us to define ``generalized Bers sections'' and ``generalized Bers embeddings'' by letting only one of the boundary components' conformal
structure vary and by looking at the resulting complex projective structure on some other (or the same) boundary component.
This idea is made precise as follows.
If an index $j \in \{1,\dots,N\}$ and marked complex structures $X_i \in {\cal T}(S_i)$ are fixed for all $i \neq j$, 
we denote by $\iota_{(X_i)}$ the 
injection
\begin{equation} \label{caninj}
\iota_{(X_i)} : 
\begin{array}{ccl}
{\cal T}(S_j) & \rightarrow & {\cal T}(\partial \hat{M})\\
  X & \mapsto & (X_1, \dots, X_{j-1},X,X_{j+1}, \dots, X_N)~.
\end{array}
\end{equation}
Let $f_{(X_i),k} = \pr_k \circ \beta \circ \iota_{(X_i)}$ as in the following diagram:
\begin{equation}
\xymatrix{
{\cal T}(\partial \hat{M}) \ar[rr]^\beta & & {\cal CP}(\partial \hat{M}) \ar[d]^{\pr_k}\\
{\cal T}(S_j) \ar[u]^{\iota_{(X_i)}}\ar[rr]^{f_{(X_i),k}} & & {\cal CP}(S_k)~.
}\end{equation}
If $j = k$, then $\sigma_{(X_i)}:=f_{(X_i),j}$ is a holomorphic section to $p_j : {\cal CP}(S_j) \to {\cal T}(S_j)$, that we call a \emph{generalized Bers
section}. On the other hand, if $j \neq k$, then $f_{(X_i),k}$ maps ${\cal T}(S_j)$ in the affine\footnote{see section \ref{affinebundle}.}
fiber $P(X_k) \subset {\cal CP}(S_k)$, we call a $f_{(X_i),k}$ a \emph{generalized Bers embedding}. We apologize for this misleading
terminology: a ``generalized Bers embedding'' is \emph{not} an embedding in general.\label{genbers}

Note that quasi-Fuchsian structures discussed in the previous paragraph just correspond to the case where $M = S \times \R$. 
Let us also mention that this discussion is easily
adapted when $\partial M$ contains tori or is no longer assumed incompressible, with a few precautions. When $\hat{M}$ only has one boundary component,
this gives the notion of a \emph{Schottky section}.
\section{The cotangent symplectic structures}\label{cotangentst}

\subsection{\texorpdfstring{$\CPS$ as an affine holomorphic bundle over $\TS$}{{CP(S)} as an affine holomorphic bundle over {T(S)}}} \label{affinebundle}

\subsubsection*{The Schwarzian derivative}

Given a locally injective holomorphic function $f: Z_1 \rightarrow Z_2$ where $Z_1$ and $Z_2$ are complex projective surfaces, 
define the \emph{osculating map} $\tilde{f}$ to $f$ at a point $m\in Z_1$
as the germ of a (locally defined) projective map that has the best possible contact with $f$ at $m$. In some sense, one can take
a flat covariant derivative $\nabla \tilde{f}$ and identify it as holomorphic quadratic differential ${\cal S}f \in Q(X)$, 
called the \emph{Schwarzian derivative} \label{defschwarzian} of $f$.
We refer to \cite{anderson} and \cite{dumas} for details.

In local projective charts, the Schwarzian derivative of $f$ has the classical expression 
${\cal S}f = Sf(z) dz^2$, where 
\begin{equation*}Sf(z) = \frac{f'''(z)}{f'(z)} - \frac{3}{2}{\left(\frac{f''(z)}{f'(z)}\right)}^2~.\end{equation*}

As a consequence of the definition, the Schwarzian operator enjoys the following properties:
\begin{propo}~\label{wwq}
\begin{itemize}
 \item[$\bullet$] If $f$ is a projective map, then ${\cal S}f = 0$ (the converse is also true).
 \item[$\bullet$] If $f: Z_1 \rightarrow Z_2$ and $g: Z_2 \rightarrow Z_3$ are locally injective holomorphic functions between complex projective surfaces,
 then \begin{equation*}{\cal S}(g\circ f) = {\cal S}(f) + f^* {\cal S}(g)~.\end{equation*}
\end{itemize}
\end{propo}

The Schwarzian derivative also satisfies an existence theorem: 
\begin{propo}\label{sct}
 If $U \subset \C$ is simply connected and $\varphi \in Q(U)$, then ${\cal S}f = \varphi$ can be solved for $f : U \rightarrow \CP$.
\end{propo}
An elementary and constructive proof of this fact is given in e.g. \cite{dumas}, see also \cite{anderson} for a more abstract argument.

\subsubsection*{Schwarzian parametrization of a fiber} \label{schwarzian}

Recall that there is a holomorphic ``forgetful'' map $p : \CPS \rightarrow \TS$.
Let $X$ be a fixed point in $\TS$ and $P(X) := p^{-1}(\{X\})$ the set of marked projective structures on $S$ whose underlying
complex structure is $X$.

Given $Z_1$, $Z_2 \in P(X)$, the identity map $\mathrm{id}_S : Z_1 \rightarrow Z_2$ is holomorphic but not projective if $Z_1 \neq Z_2$. 
Taking
its Schwarzian derivative accurately measures the ``difference'' of the two projective structures $Z_1$ and $Z_2$.
Let us make this observation more precise. A consequence of Proposition \ref{sct} is that given $Z_1\in P(X)$ and $\varphi\in Q(X)$,
there exists $Z_2 \in P(X)$ such that ${\cal S}\left(\mathrm{id}_S : Z_1 \rightarrow Z_2\right) = \varphi$. This defines a map $Q(X) \times P(X) \rightarrow P(X)$,
which is now easily seen to be a freely transitive action of $Q(X)$ on $P(X)$ as a consequence of Proposition \ref{wwq}.
In other words, $P(X)$ is equipped with a complex affine structure, modeled on the vector space $Q(X)$.

Recall that $Q(X)$ is also identified with the complex dual space $T_X^*\TS$. As a result of this discussion, 
$\CPS$ is an affine holomorphic bundle modeled on the holomorphic cotangent vector bundle $T^* \TS$.

Consequently, $\CPS$ can be identified with $T^* \TS$ by choosing a ``zero section'' $\sigma : \TS \to \CPS$.
Explicitly, we get an isomorphism of complex affine bundles $\tau^{\sigma} : Z \mapsto Z - \sigma\left(p(Z)\right)$ 
as in the following diagram:

\begin{equation}\label{tautau}\xymatrix{
\CPS \ar[rd]_{p} \ar[rr]^{\tau^\sigma} & & T^* \TS \ar[ld]^{\pi} \\
 & \TS~. &
}\end{equation}
$\tau^\sigma$ is characterized by the fact that $\tau^\sigma \circ \sigma$ is the zero section to $\pi : T^* \TS \to \TS$.
It is an isomorphism of holomorphic bundles whenever $\sigma$ is a holomorphic section to $p$,
such as a (generalized) Bers section (see sections \ref{berssection} and \ref{genbers}).

\subsection{\texorpdfstring{Complex symplectic structure on $T^*\TS$}{Complex symplectic structure on T*T(S)}}\label{compsympcot}

It is a basic fact that if $M$ is any complex manifold (in particular when $M=\TS$), the total space of its 
holomorphic cotangent bundle $T^*M$
\footnote{In this context $T^*M$ stands for the complex dual of the holomorphic tangent bundle $T^{1,0}M$, its (smooth) sections 
are the $(1,0)$-forms.} is equipped
with a canonical complex symplectic structure. We briefly recall this and a few useful properties.

The \emph{canonical $1$-form $\xi$} \label{defcanform} is the holomorphic $(1,0)$-form on $T^*M$ defined at a point $\varphi \in T^*M$ by 
${\xi}_\varphi := \pi^* \varphi$, where $\pi : T^*M \to M$ is the canonical projection and $\varphi$ is seen as a complex covector on $M$
in the right-hand side of the equality. The \emph{canonical complex symplectic form} on $T^*M$ is then simply defined by 
$\omega_{\mathrm{can}} = d\xi$
\footnote{Note that some authors might take the opposite sign convention for $\omega_{\mathrm{can}}$.}.
If $(z_k)$ is a system of holomorphic coordinates on $M$ so that an arbitrary $(1,0)$-form has an expression of the form
$\alpha = \sum w_k dz_k$, then $(z_k,w_k)$ is a system of holomorphic coordinates on $T^*M$ for which
$\xi = \sum w_k dz_k$ and $\omega_{\mathrm{can}} = \sum dw_k \wedge dz_k$.

The canonical $1$-form satisfies the following reproducing property. If $\alpha$ is any $(1,0)$-form on $M$, it is in particular
a map $M \to T^*M$ and as such it can be used to pull back differential forms from $T^*M$ to $M$. It is then not hard to show that 
$\alpha^* \xi = \alpha \label{reprod}$
and as a consequence $\alpha^* \omega_{\mathrm{can}} = d\alpha\label{reprod2}$.

Note that if $u$ is a vertical tangent vector to $T^*M$, {\it i.e.} $\pi_* u = 0$, then $u$ can be identified with an element 
of the fiber containing
its base point $\alpha$ (since the fibers of the projection are vector spaces).
Under that identification, the symplectic pairing between $u$ and any other tangent vector $v \in T_\alpha T^* M$ is just given by 
$\omega_{\mathrm{can}}(u,v) = \langle u,\pi_*v \rangle$ where $\langle\cdot,\cdot\rangle$ is the duality pairing
on $T_{\pi(\alpha)} M$.

Note that the fibers of the projection $\pi : T^*M \rightarrow M$ are Lagrangian
submanifolds of $T^*M$, in other words $\pi$ is a Lagrangian fibration. Also, the zero section $s_0 : M \hookrightarrow T^*M$ \label{defzerosection} is a 
Lagrangian embedding. These are direct consequences of the previous observation.

\subsection{Affine cotangent symplectic structures}\label{affid}

As we have seen in section \ref{schwarzian}, any choice of a ``zero section'' $\sigma : \TS \to \CPS$ yields an affine isomorphism 
$\tau^{\sigma} : \CPS \stackrel{\sim}{\to} T^* \TS$. We can use this to pull back the canonical symplectic structure of $T^*\TS$ on $\CPS$:
define \begin{equation}\omega^\sigma := (\tau^{\sigma})^* \omega_{\mathrm{can}}~.\end{equation}
It is clear that $\omega^\sigma$ is a complex symplectic form on $\CPS$ whenever $\sigma$ is a holomorphic section to $p$. Otherwise, it is just
a complex-valued non-degenerate $2$-form on $\CPS$, whose real and imaginary parts are both real symplectic forms.

How is $\omega^\sigma$ affected by the choice of the ``zero section'' $\sigma$? The following statement is both straightforward and key:
\begin{propo}\label{craca}
 For any two sections $\sigma_1$ and $\sigma_2$ to $p: \CPS \to \TS$,
\begin{equation}\omega^{\sigma_2} - \omega^{\sigma_1} = - p^* d(\sigma_2 - \sigma_1)\end{equation}
\end{propo}
\noindent where $\sigma_2 - \sigma_1$ is the ``affine difference'' between $\sigma_2$ and $\sigma_1$, it is a $1$-form on $\TS$. In particular, the symplectic
structures induced by the respective choices of two sections agree if and only if their affine difference is a closed $1$-form.

\begin{proof}
 This is an easy computation:
\begin{eqnarray*}
  - p^* d\left(\sigma_2 - \sigma_1\right) & = & -p^* \left((\sigma_2 - \sigma_1)^*\omega_{\mathrm{can}}\right) ~~ \textrm{(see (\ref{reprod2}))}\\
  & = & \left(-(\sigma_2 - \sigma_1)\circ p\right)^*\omega_{\mathrm{can}}\\
  & = & \left(\tau^\sigma_2 - \tau^\sigma_1\right)^*\omega_{\mathrm{can}}\\
  & = & \left(\tau^{\sigma_2}\right)^*\omega_{\mathrm{can}} - \left(\tau^{\sigma_1}\right)^*\omega_{\mathrm{can}}~.
\end{eqnarray*}
Only the last step is not so trivial as it would seem because one has to be careful about base points.
Also, note that in the identity 
$$\tau^{\sigma_2}(Z) - \tau^{\sigma_1}(Z) = \left(Z - \sigma_2 \circ p (Z)\right) - \left(Z - \sigma_1 \circ p (Z)\right) = - (\sigma_2 - \sigma_1)\circ p(Z)~,$$
some minus signs are ``affine'' ones (hiding the Schwarzian derivative) and others are ``genuine'' minus signs, 
but this can be ignored in computation.
\end{proof}

Moreover, a straightforward calculation gives an explicit expression of $\omega^\sigma(u,v)$ whenever $u$ is a vertical tangent vector to $\CPS$, 
it is exactly the same as the one obtained for the symplectic structure on $T^*\TS$:

\begin{propo}\label{vertg}
Let $\sigma : \TS \to \CPS$ be a section to $p$. Let $Z$ be a point in $\CPS$, and $u$, $v$ be tangent vectors at $Z$ such that $u$ is vertical, {\it i.e.}
$p_* u= 0$. Then  \begin{equation}\omega^\sigma(u,v) = \langle u,p_*v \rangle~.\end{equation}
\end{propo}
In this expression, $u$ is seen as an element of $\in {T_X}^* T(S)$ (where $X = p(Z)$) under the identification $T_Z P(X) = Q(X) = {T_X}^* T(S)$.
Note that this expression not involving $\sigma$ is compatible with the previous proposition, which implies that 
$\omega^{\sigma_2} - \omega^{\sigma_1}$ is a horizontal $2$-form.

As a consequence, just like in the cotangent space, we have:
\begin{propo}
Let $\sigma : \TS \to \CPS$ be any section. The projection $p : \CPS \to \TS$ is a Lagrangian fibration for $\omega^\sigma$.
Also, $\sigma$ is a Lagrangian embedding.
\end{propo}

\section{The character variety and Goldman's symplectic structure}\label{charGs}

\subsection{The character variety}\label{charvar}

References for this section include \cite{goldmannature}, \cite{porti}, \cite{Gsl2} and \cite{dumas}.

Let $G = \PSL_2(\C)$ and $\RS$ be the set of group homomorphisms from $\pi := \pi_1(S)$ to $G$.
It has a natural structure of a complex affine algebraic set as follows. Choose a finite presentation 
$\pi = \langle \gamma_1, \dots, \gamma_N ~|~ (r_i)_{i\in I} \rangle$
of $\pi$. Evaluating a representation $\rho\in\RS$ on the generators $\gamma_k$ embeds ${\cal R(S)}$ as an algebraic subset of $G^N$.
This gives $\RS$ an affine structure indeed because of the identification $\PSL_2(\C) \approx \SO_3(\C)$ (given by the adjoint representation of
$\PSL_2(\C)$ on its Lie algebra $\mathfrak{g} = \slC$). One can check that this structure is independent of the presentation.

$G$ acts algebraically on $\RS$ by conjugation. The character variety ${\cal X}(S)$ is defined as the quotient in the sense of invariant theory.
Specifically, the action of $G$ on $\RS$ induces an action on the ring of regular functions $\C[\RS]$. Denote by $\C[\RS]^G$ the ring of invariant
functions, it is finitely generated because $\RS$ is affine and $G$ is reductive.
\begin{lemma}[see e.g. \cite{porti}]
 In fact, it is generated in this case ($G=\PSL_2(\C)$) by a finite number of the complex valued functions on $\RS$ of the form 
$\rho \mapsto \mathrm{tr}^2(\rho(\gamma))$.
\end{lemma}

$\XS$ is the affine set such that $\C[\XS] = \C[\RS]^G$, it is called the \emph{character variety} of $S$.
A consequence of the lemma is that the points of $\XS$ are in one-to-one correspondence with
the set of \emph{characters}, {\it i.e.} complex-valued functions of the form $\gamma \in \pi \mapsto \mathrm{tr}^2(\rho(\gamma))$.

The affine set $\XS$ splits into two irreducible components $\XS_l \cup \XS_r$, where elements of $\XS_l$
are characters of representations that lift to $\SL(2,\C)$.

The set-theoretic quotient $\RS /G$ is rather complicated, but $G$ acts freely and properly on the subset $\RS^s$ of irreducible\footnote{
A representation $\rho : \pi \to \PSL_2(\C)$ is called \emph{irreducible} if it fixes no point in $\CP$.} (``stable'') 
representations, so that the
quotient $\RS^s/G$ is a complex manifold.
Furthermore, an irreducible representation is determined by its character, so that $\XS^s := \RS^s/G$ embeds (as a Zariski-dense open subset) in the smooth locus
of $\XS$. Its dimension is $6g-6$. Let us mention that more generally, $\XS$ is in bijection with the set of orbits of ``semistable'' ({\it i.e.} reductive
\footnote{A nontrivial representation $\rho : \pi \to \PSL_2(\C)$ is called \emph{reductive} if it is either irreducible of it fixes a pair of distinct points 
in $\CP$.})
representations.

It is relatively easy to see that the Zariski tangent space at a point $\rho \in \RS$ is described as 
the space of crossed homomorphisms $Z^1(\pi, \mathfrak{g}_{\mathrm{Ad}\circ\rho})$
 ({\it i.e.} $1$-cocycles in the
sense of group cohomology), specifically maps $u : \pi \rightarrow \slC$ such that 
$u(\gamma_1 \gamma_2) = u(\gamma_1) + \mathrm{Ad}_{\rho(\gamma_1)}u(\gamma_2)$ \footnote{where of course $\mathrm{Ad} : G \to \mathrm{Aut} \mathfrak{g}$
is the adjoint representation. \label{adjointdef}}.
The subspace corresponding to the tangent space of the $G$-orbit of $\rho$ is the
space of principal crossed homomorphisms $B^1(\pi, \mathfrak{g}_{\mathrm{Ad}\circ\rho})$ ({\it i.e.} $1$-coboundaries in the
sense of group cohomology), specifically maps $u : \pi \rightarrow \slC$ such that 
$u(\gamma) = \mathrm{Ad}_{\rho(\gamma)}u_0 - u_0$ for some $u_0 \in\slC$. Hence for (at least) smooth points $[\rho] \in \XS$, the tangent space
is given by $T_{[\rho]}\XS = H^1(\pi,\mathfrak{g}_{\mathrm{Ad}\circ\rho})$.

\subsection{The complex symplectic structure on the character variety}

By the general construction of \cite{goldmannature}, the character variety enjoys a complex symplectic structure defined in this situation as follows.

Recall that the Lie algebra $\mathfrak{g} = \slC$ is equipped with its complex Killing form $B$ \label{killingdef}. It is a non-degenerate complex bilinear symmetric form preserved
by $G$ under the adjoint action. Let $\tilde{B}=\frac{1}{4} B$, it is explicitly given by $\tilde{B}(u,v) = \mathrm{tr}(uv)$ where $u, v\in \slC$ 
are represented by trace-free $2 \times 2$ matrices.

One can compose the standard cup-product in group cohomology with $\tilde{B}$
\footnote{It would look somewhat more natural to use the actual Killing form $B$ instead of $\tilde{B} = \frac{1}{4}B$, but we choose to go 
with $\tilde{B}$ because it is the convention used by most authors. Moreover, it gives a slightly simpler expression of our theorems \ref{main1},
\ref{main2} and \ref{main3}.} as ``coefficient pairing'' to get a dual pairing
\begin{equation}H^1(\pi,\mathfrak{g}_{\textrm{Ad} \circ \rho}) \times H^1(\pi,\mathfrak{g}_{\textrm{Ad} \circ \rho}) \stackrel{\cup}{\to}
H^2(\pi,\mathfrak{g}_{\textrm{Ad} \circ \rho} \otimes \mathfrak{g}_{\textrm{Ad} \circ \rho}) \stackrel{\tilde{B}}{\to} H^2(\pi,\C) \cong \C ~.\end{equation}

This pairing defines a non-degenerate complex bilinear alternate $2$-form on
the complex vector space $H^1(\pi,\mathfrak{g}_{\textrm{Ad} \circ \rho}) \approx T_{[\rho]}\XS$.
It globalizes into a non-degenerate $2$-form $\omega_G$ on $\XS^s$. By arguments of Goldman (\cite{goldmannature}) following Atiyah-Bott (\cite{AB})
this form is closed, in other words it is a complex symplectic form on the smooth quasi-affine variety $\XS^s$ \label{gsdef}
\footnote{In fact, it defines an algebraic tensor on the whole character variety, see \cite{goldmannature}.}.

\subsection{Holonomy of projective structures}\label{holonomy}

Just like any geometric structure, a complex projective structure $Z$ defines a \emph{developing map}
and a \emph{holonomy representation} (see e.g. \cite{thurston}). The developing map is a locally injective projective map 
$f : \tilde{Z} \rightarrow \CP$ 
and it is equivariant with respect to the holonomy representation $\rho : \pi \rightarrow \PSL_2(\C)$ in the sense that 
$f \circ \gamma = \rho(\gamma) \circ f$ for any 
$\gamma \in \pi$.

Holonomy of complex projective structures defines a map
\begin{equation*}
\hol : \CPS \rightarrow \XS~.
\end{equation*}
It is differentiable and its differential is ``the identity map'' in the sense that 
it is the canonical identification 
$$d\hol : T_Z\CPS = \check{H}^1(Z,\Xi_Z) \stackrel{\sim}{\rightarrow} H^1(\pi,\mathfrak{g}_{\textrm{Ad} \circ \rho}) = T_{[\hol(Z)]}\XS~.$$
A consequence of this observation is that $\hol$ is a local biholomorphism.

The holonomy representation $\rho$ of a complex projective structure satisfies the following properties:
\begin{itemize}
 \item[$\bullet$] $\rho$ is liftable to $\SL_2(\C)$ (a lift is provided by the monodromy of the Schwarzian equation).
 The image of the holonomy map thus lies in the irreducible component $\XS_l$ of $\XS$.
 \item[$\bullet$] The action of $\Gamma := \rho(\pi)$ on hyperbolic $3$-space $\HH^3$ does not fix any point or ideal point, 
 nor does it preserve any geodesic. Representations having this property are called non-elementary. They are in particular irreducible representations, 
hence smooth points of the character variety as expected.
\end{itemize}
Conversely, {Gallo}-{Kapovich}-{Marden} showed that any non-elementary liftable representation is the holonomy of a complex
projective structure (\cite{gallo}). 

Although the holonomy map $\hol : \CPS \rightarrow \XS$ is a local biholomorphism, it is neither injective nor a covering onto its image (\cite{hejhal}).
Nonetheless, we get a complex symplectic structure on $\CPS$ simply by pulling back that of $\XS^s$ by the holonomy map. Abusing notations, we will still
call this symplectic structure $\omega_G$. Alternatively, one could directly define $\omega_G$ on $\CPS$ in terms of the exterior product of $1$-forms with 
values in some flat bundle
(recall that $T_Z \CPS = \check{H}^1(Z,\Xi_Z)$, where $\Xi_Z$ is the sheaf of projective vector fields on $Z$, see section \ref{pif}). 
We will consider $\omega_G$ as the \emph{standard complex symplectic structure on $\CPS$} (notably because it does not depend on 
any choice).

\subsection{Fuchsian structures and a theorem of Goldman}\label{fuchhol}

Let $\FS$ be the space of marked hyperbolic structures on $S$ (we abusively use the same notation as for the Fuchsian space). More precisely, $\FS$
is the space of complete hyperbolic metrics on $S$ quotiented by $\mathrm{Diff}^+_0(S)$. 
In terms of geometric structures, $\FS$ is the 
deformation space of $(\HH^2,\PSL_2(\R))$-structures on $S$ (this is a consequence of Cartan-Hadamard's theorem). 
Holonomy identifies $\FS$ as the connected component of the character variety ${\cal X}(S,\PSL_2(\R))$ corresponding 
to faithful and discrete representations. $\FS$ is sometimes called the Fricke space of $S$.

The uniformization theorem states that there is a unique hyperbolic metric in each conformal class of Riemannian metrics on $S$. 
Since $S$ is oriented, the choice of a conformal structure on $S$ is equivalent to that of a complex structure on $S$. 
The uniformization theorem thus
provides a bijective map $$u : \TS \to \FS~.$$
By definition of the Fuchsian section $\sigma_{\cal F}$, the map $u$ is precisely identified as $\sigma_{\cal F}$ if hyperbolic structures are considered
as special examples of complex projective structures. Putting it differently, the following diagram commutes:
\begin{equation}\xymatrix{
\CPS \ar[rr]^{\hol} & & {\cal X}(S,\PSL_2(\C))\\
\TS \ar[rr]^{u} \ar[u]^{\sigma_{\mathcal{F}}} & & \FS \hookrightarrow {\cal X}(S,\PSL_2(\R))~. \ar[u]^\iota
}\end{equation}
It is derived from this diagram that $\sigma_{\cal F}$ is a maximal totally real\footnote{see 
section \ref{ancont} for a definition of this notion and a different argument for this fact.} 
analytic embedding of $\TS$ in $\CPS$.

By Goldman's general construction in \cite{goldmannature} (described above 
in the case of $G =\PSL_2(\C)$),
${\cal X}(S,\PSL_2(\R))$ is equipped with a real symplectic structure $\omega_{G,\PSL_2(\R)}$. Of course
 it is just the restriction of the symplectic structure $\omega_G = \omega_{G,\PSL_2(\C)}$ on ${\cal X}(S,\PSL_2(\R))$.
Recall that $\TS$ is also equipped with a symplectic structure, the Weil-Petersson K\"ahler form $\omega_{WP}$. 
In the same paper, Goldman shows that they are the same. More precisely, this is expressed in our setting as follows:
\begin{theorem}[Goldman \cite{goldmannature}]
\begin{equation}\label{wgfsss}
 {(\sigma_{\cal F})}^* \omega_{G} = \omega_{WP}~.
\end{equation}
\end{theorem}

\subsection{A Lagrangian embedding}

Let $\hat{M}$ be a compact $3$-manifold as in section \ref{cphyp}. We will use here the same notations as in section \ref{cphyp}, let us briefly recall these.
The boundary $\partial \hat{M}$ is the disjoint
union of $N$ surfaces $S_k$ of genera at least $2$. The Teichm\"uller space of the boundary is given by 
${\cal T}(\partial \hat{M}) = {\cal T}(S_1) \times \dots \times {\cal T}(S_N)$, and similarly 
${\cal CP}(\partial \hat{M}) = {\cal CP}(S_1) \times \dots \times {\cal CP}(S_N)$. The forgetful projection is the holomorphic map 
$p = p_1 \times \dots \times p_N : {\cal CP}(\partial \hat{M}) \to {\cal T}(\partial \hat{M})$, and 
$\beta : {\cal T}(\partial \hat{M}) \to {\cal CP}(\partial \hat{M})$ is the ``simultaneous
uniformization section''.

By Goldman's construction discussed above, ${\cal CP}(\partial \hat{M})$ is equipped with a complex symplectic structure $\omega_G$, which is obtained here as
\begin{equation} \omega_G = {\pr_1}^* {\omega_G}^{(1)} + \dots + {\pr_N}^* {\omega_G}^{(N)}~, \end{equation}
where ${\omega_G}^{(k)}$ is the complex symplectic structure on ${\cal CP}(S_k)$
and  $\pr_k$ is the $k^{\textrm{th}}$ projection map ${\cal CP}(\partial \hat{M}) \to {\cal CP}(S_k)$.

There is a general argument, discovered in this setting by Steven Kerckhoff, which shows that
\begin{theorem}\label{lagemb}
$\beta : {\cal T}(\partial \hat{M}) \to {\cal CP}(\partial \hat{M})$
is a Lagrangian embedding. 
\end{theorem}
Although this is
a consequence of our theorem \ref{puuff}, we briefly explain this nice argument, based on Poincar\'e duality in cohomology.
This could be done directly on the manifolds ${\cal HC}(M)$ and ${\cal CP}(\partial \hat{M})$, but we prefer
to transport the situation to character varieties, where it is simpler.

Recall that the simultaneous uniformization 
section $\beta$ was defined as the composition $\beta = \psi \circ \varphi^{-1}$, where $\psi : {\cal HC}(M) \to {\cal CP}(\partial \hat{M})$ is the map
which assigns the induced projective structure on $\partial \hat{M}$ to each cocompact hyperbolic structure on the interior $M$ of $\hat{M}$, and 
$\varphi = p \circ \psi : {\cal HC}(M) \to {\cal T}(\partial \hat{M})$ is a biholomorphism. By definition, the embedding $\beta$ is Lagrangian if it is isotropic
 ($\beta^*\omega_G = 0$) and $\dim {\cal CP}(\partial \hat{M}) = 2 \dim {\cal T}(\partial \hat{M})$. We already know the second
statement to be true (see section \ref{pif}). It remains to show that $\beta$ is isotropic, but since $\phi$ is a diffeomorphism, this amounts
to showing that $\psi : {\cal HC}(M) \to {\cal CP}(\partial \hat{M})$ is isotropic.

 Let us have a look at the equivalent statement on holonomy: there is a commutative
diagram $$\xymatrix{
{\cal HC}(M) \ar[d]^{\hol} \ar[rr]^\psi& & \ar[d]^{\hol} {\cal CP}(\partial \hat{M})\\
{\cal \hat{X}}(M,\PSL_2(\C)) \ar[rr]^{f} & & {\cal X}(\partial \hat{M},\PSL_2(\C))~,
}$$ where $f : {\cal X}(M,\PSL_2(\C)) \to {\cal X}(\partial \hat{M},\PSL_2(\C))$ is the map between character varieties induced by
the ``restriction'' map $\iota_* : \pi_1(\partial \hat{M}) \to \pi_1(\hat{M})$\footnote{Note that if $\partial \hat{M}$ is disconnected, we define
its fundamental group $\pi_1(\partial \hat{M})$ as the free product of the fundamental groups of its components, so that a representation $\rho : \pi_1(\partial 
\hat{M}) \to \PSL_2(\C)$ is just a $N$-tuple of representations $\rho_k : \pi_1(S_k) \rightarrow \PSL_2(\C)$.}. Since the property of being isotropic is local,
it is enough to show the following proposition:

\begin{propo}\label{holag}
The map $f : {\cal X}(\hat{M},\PSL_2(\C)) \to {\cal X}(\partial \hat{M},\PSL_2(\C))$ is isotropic.
\end{propo}
\begin{proof}
Let $[\rho] \in {\cal X}(\hat{M},\PSL_2(\C))$. The map 
$$df : T_{[\rho]} {\cal X}(\hat{M},\PSL_2(\C) \to T_{[\rho \circ \iota_*]} {\cal X}(\partial \hat{M},\PSL_2(\C))$$
is the map $\alpha$ that appears in long exact sequence in cohomology of the pair $(M, \partial M)$ as follows. This exact diagram shows a piece of this
sequence written in terms of group cohomology, where vertical arrows are given by Poincar\'e duality:
\begin{small}
$$\xymatrix{
H^1(\pi_1(\hat{M}),\mathfrak{g}_{\mathrm{Ad}\circ\rho}) \ar[r]^\alpha \ar[d] & H^1(\pi_1(\partial \hat{M}),\mathfrak{g}_{\mathrm{Ad}\circ\rho}) \ar[r]^\beta \ar[d] & 
H^2(\pi_1(\hat{M}),\pi_1(\partial \hat{M}); \mathfrak{g}_{\mathrm{Ad}\circ\rho}) \ar[d]\\
H^2(\pi_1(\hat{M}),\pi_1(\partial \hat{M}) \ar[r]^{\beta^*}; \mathfrak{g}_{\mathrm{Ad}\circ\rho})^* & 
H^1(\pi_1(\partial \hat{M}),\mathfrak{g}_{\mathrm{Ad}\circ\rho})^* \ar[r]^{\alpha^*} & 
H^1(\pi_1(\hat{M}),\mathfrak{g}_{\mathrm{Ad}\circ\rho})^*
}$$
\end{small}
Note that if $u\in H^1(\pi_1(\partial \hat{M}),\mathfrak{g}_{\mathrm{Ad}\circ\rho})$, 
the Poincar\'e dual of $u$ is defined by the relation 
$\langle u^*, v \rangle = \tilde{B}(u \cup v) \cap [\partial \hat{M}]$ 
for all $v \in H^1(\pi_1(\partial \hat{M}),\mathfrak{g}_{\mathrm{Ad}\circ\rho})$, 
where $[\partial \hat{M}]$ is the fundamental class of $\partial{\hat M}$. This is precisely
saying that $\langle u^*, v \rangle = \omega_G(u,v)$. It follows that $\alpha$ is isotropic: using the commutativity and
exactness of the diagram, we can write
\begin{eqnarray*}
 \omega_G(\alpha(u),\alpha(v)) & = & \langle \alpha(u)^*,\alpha(v) \rangle\\
 & = & \langle \beta^*(u^*),\alpha(v) \rangle\\
 & = & \langle u^*,\beta \circ \alpha(v) \rangle\\
 & = & 0~.
\end{eqnarray*}
\end{proof}

\begin{remark}
Note that in the quasi-Fuchsian situation $M = S \times \R$, Theorem \ref{lagemb} is trivial, or rather its formulation in terms of
holonomy (cf. Proposition \ref{holag} above). Indeed, the map $f : {\cal X}(\hat{M},\PSL_2(\C)) \to {\cal X}(\partial \hat{M},\PSL_2(\C))$
in that case is just the diagonal
embedding of ${\cal X}(\pi,\PSL_2(\C))$\footnote{where $\pi=\pi_1(\hat{M})=\pi_1(S)$.} into 
${\cal X}(\pi,\PSL_2(\C)) \times {\cal X}(\pi,\PSL_2(\C))$\footnote{Here ${\cal X}(\pi,\PSL_2(\C)) \times {\cal X}(\pi,\PSL_2(\C))$ is
equipped with with the complex symplectic structure ${\pr_1}^*\omega_G-{\pr_2}^*\omega_G$ 
(the minus sign is due to the opposite orientation of $\partial^+ \hat{M}$ and $\partial^- \hat{M}$). The fact that the diagonal
is Lagrangian is a particular case of the following general fact: if $(X,\omega)$ is a symplectic
manifold and $X\times X$ is equipped with the symplectic structure ${\pr_1}^*\omega-{\pr_2}^*\omega$, then the graph of a function $h : X\to X$ is a
Lagrangian submanifold of $X\times X$ if and only if $h$ is a symplectomorphim.}. 
\end{remark}

\section{Complex Fenchel-Nielsen coordinates and Platis' symplectic structure} \label{complexFNP}

\subsection{Fenchel-Nielsen coordinates on Teichm\"uller space and Wolpert theory}\label{reio}

\subsubsection*{Pants decomposition and Fenchel-Nielsen coordinates}

In this paragraph, we consider the Fuchsian (or Fricke) space 
$\FS$ of marked hyperbolic structures on $S$ (or marked Fuchsian projective structures)
rather than Teichm\"uller space of $\TS$ (see section \ref{fuchhol}).
Let us first briefly recall the construction of the classical Fenchel-Nielsen coordinates on $\FS$, as it will be useful
for the subsequent paragraphs.
These depend on the choice of a \emph{pants decomposition} \label{pdc} of $S$, {\it i.e.} an ordered maximal 
collection of distinct, disjoint
\footnote{in the sense that for $j\neq k$, there exists disjoint representatives of $\alpha_j$ and $\alpha_k$.}, 
nontrivial free homotopy classes of simple\footnote{meaning that there exist simple representatives.} 
closed curves $\alpha = (\alpha_1, \dots, \alpha_N)$.

The following are classical facts:
\begin{itemize}
 \item[$\bullet$] $N = 3g-3$.
 \item[$\bullet$] If $c_1, \dots, c_N$ are disjoint representatives of $\alpha_1, \dots, \alpha_N$ (respectively), then $S \setminus \bigcup_{i=1}^N c_i$
 is a disjoint union of $M = 2g-2$ topological pair of pants $P_k$ (thrice-punctured spheres).
 \item[$\bullet$] If $X$ is a hyperbolic structure on $S$, every nontrivial free homotopy class of simple closed curves $\gamma$
 is uniquely represented by a simple a closed geodesic
 $\gamma^X$.
\end{itemize}

Given a hyperbolic structure $X$ on $S$, denote by $l_{\gamma}(X)$ the hyperbolic length of $\gamma^X$. 
This defines a \emph{length function}
 $l_{\gamma} : \FS \to \R_{>0}$. In particular, given a pants decomposition $\alpha$, one gets a function 
$l_\alpha : \FS \to {(\R_{>0})}^N~$. \label{lenfun}
The components $l_{\alpha_i}$ of $l_\alpha$ are called the \emph{Fenchel-Nielsen length parameters}.

Any hyperbolic structure $X$ on $S$ induces a hyperbolic structure (with geodesic boundary) on each 
one of the closed pair of pants $\overline{P_k}$ in the decomposition
$S\setminus \bigcup_{i=1}^N \alpha_i^X = \bigsqcup_{k=1}^M P_k$.
It is well-known that a hyperbolic structure on a closed pair of pants is uniquely determined by the lengths of its three boundary components.
This follows from the observation that a hyperbolic pair of pants is obtained by gluing two isometric oppositely oriented right-angled hexagons in $\HH^2$
and the following elementary theorem in plane hyperbolic geometry:
\begin{propo}
Up to isometry, there exists a unique right-angled hexagon in $\HH^2$ with prescribed lengths on every other side.
\end{propo}

As a consequence, a hyperbolic structure on $S$ is completely determined by the lengths of the curves $\alpha_i$, and the parameters $\tau_i$ that prescribe how
the gluing occurs along these curves, {\it i.e.} by which amount of ``twisting''.
However, these parameters $\tau_i$ are not very well defined: 
there is no obvious choice of the hyperbolic structure obtained by ``not twisting at all before gluing''. Also, note that assuming that such
a choice is made, each of these parameters should live in $\R$ indeed and not $\R/2\pi\Z$: although there is a natural isometry $f:X \to Y$ where $Y$ is obtained
by $2\pi$-twisting $X$ along some curve $\alpha_i$, $f$ is not homotopic to the identity.

Let us make this more precise. For any nontrivial free homotopy class of simple closed curves $\gamma$, there is a flow (an $\R$-action)
called \emph{twisting along $\gamma$}
\begin{equation}\mathrm{tw}_\gamma : \R \times \FS \to \FS~.\end{equation} \label{twistflow}
The flow is freely transitive in the fibers of $l_\gamma$. Let us mention that
twist deformations along simple closed curves are naturally generalized first to weighted multicurves, then to the completion ${\cal ML}(S)$ of measured
laminations. This generalization is the notion of \emph{earthquake} introduced by Thurston (see e.g. \cite{kerck}).

Denote by $\mathrm{tw}_\alpha$ the
$\R^N$-action $\mathrm{tw}_\alpha = (\mathrm{tw}_{\alpha_1}, \dots, \mathrm{tw}_{\alpha_N}) : \R^N \times \FS \to \FS$.
The fact that a hyperbolic structure on $S$ is uniquely determined by the lengths parameters $l_{\alpha_i}$ and 
the amount of twisting a long each $\alpha_i$
is precisely stated as: the $\R^N$-action $\mathrm{tw}_\alpha$ is freely transitive in the fibers of $l_\alpha$, 
and the reunion of these fibers is the whole
Fricke space $\FS$. In particular,
\begin{theorem}\label{FN}
Choosing a smooth section to $l_\alpha$ determines a diffeomorphism $$(l_\alpha,\tau_\alpha) :  \FS \to {(\R_{>0})}^N \times \R^N~.$$
\end{theorem}
The function $\tau$ above is naturally defined by $\mathrm{tw}_\alpha(\tau_\alpha,\sigma \circ l_\alpha) = \mathrm{id}_{\FS}$, where 
$\sigma$ is the chosen section. The components $\tau_{\alpha_1}, \dots, \tau_{\alpha_N}$ of $\tau$ are called the \emph{Fenchel-Nielsen twist parameters}.
The theorem above thus says that Fenchel-Nielsen length and twist parameters are global coordinates on $\FS$. 
In particular, one recovers $\dim_\R \TS = \dim_\R \FS = 2N =6g-6$.
It also appears that $\TS \approx \FS$ is topologically a cell, and it follows that $\CPS$ is also a cell.

Note that although the coordinates $\tau_{\alpha_i}$ depend on the choice on a section to $l_\alpha$, the $1$-forms
$d\tau_{\alpha_i}$ and the vector fields $\frac{\partial}{\partial {\tau_{\alpha_i}}}$
do not. In fact, $\frac{\partial}{\partial {\tau_\gamma}}$ is well-defined for any nontrivial free homotopy class of simple closed curve
$\gamma$, and its flow
is of course the twist flow $\mathrm{tw}_{\gamma} : \R \times \FS \to \FS$.

\subsubsection*{Wolpert theory}

We recall a few notions of symplectic geometry and the language of Hamiltonian mechanics. If $(M^{2N},\omega)$ is a symplectic manifold,
$\omega$ determines an bundle map $\omega^\flat : TM \to T^*M$ defined by $\omega^\flat(u) = \omega(u,\cdot)$. Since $\omega$ is non-degenerate, 
$\omega^\flat$ is an isomorphism,
its inverse is denoted by $\omega^\sharp : T^*M \to TM$. If $\alpha$ is a one-form on $M$, $\omega^\sharp(\alpha)$ is thus the unique vector field $X$ such that 
$i_X \omega = \alpha$. If $f$ is a function on $M$, the vector field $X_f := \omega^\sharp(df)$ is called the \emph{Hamiltonian} 
\label{hamildef} (or symplectic gradient) of $f$.
Note that a vector field $X$ is Hamiltonian is and only if the $1$-form $i_X \omega$ \footnote{
where $i_X \omega$ is the contraction of $\omega$ with the vector field $X$.} is exact, it follows that $X$ satisfies ${\cal L}_X \omega = 0$
\footnote{where ${\cal L}_X$ is the Lie derivative along the vector field $X$.}
by Cartan's magic formula. Vector fields $X$ such that ${\cal L}_X \omega = 0$ are the vector fields whose flows preserve $\omega$, 
they are called symplectic vector fields. The \emph{Poisson bracket} of two functions $f$ and $g$ is defined by $\{f,g\} = \omega(X_f,X_g)$.
$f$ and $g$ are said to Poisson-commute (or to be in involution) if $\{f,g\} = 0$. It is easy to see that $f$ and $g$ Poisson-commute if and only
if $f$ is constant along the integral curves of $X_g$ (and vice-versa). If $f = (f_1,\dots,f_N) : M \to \R^N$ is a regular
map such that the $f_i$ Poisson-commute, then $f$ is a Lagrangian fibration. Moreover, the flows of the $-X_{f_i}$ (if they are complete) 
define a transitive $\R^N$-action 
that is transverse to the fibers of $f$ (the reason for the choice of this minus sign will be apparent shortly).
Notice already the analogy with the lengths functions and twist flows above. 
Such functions $f_i$ are said to define a \emph{(completely) integrable Hamiltonian system} on $(M,\omega)$.
As in theorem \ref{FN}, choosing a section to $f$ yields coordinates $g = (g_1,\dots,g_N) : M \to \R^N$\footnote{To be accurate, $g$ takes
values in $\R^{N-k} \times \mathbb{T}^k$ in general, where $k$ is some integer and $\mathbb{T}^k$ is the $k$-dimensional torus.} such
that the $\R^N$-action is given by the flows of the $\frac{\partial}{\partial g_i}$, 
in other words $\frac{\partial}{\partial g_i} = -X_{f_i}$.
In general though, $(f_i,g_i)$ is not a system of Darboux 
coordinates\footnote{By definition, $(f_i,g_i)$ are called \emph{Darboux coordinates}
on $(M,\omega)$ if they are canonical for the symplectic structure: $\omega=\sum_{i=1}^N df_i \wedge dg_i$. 
The celebrated theorem of Darboux
says that there always exists Darboux coordinates locally on any symplectic manifold.} for $\omega$, 
but the classical Arnold-Liouville theorem states
that such a choice of coordinates is possible in a way that is compatible with the Lagrangian fibration and the $\R^N$-action 
(see e.g. \cite{Duis} for
a precise statement and proof of this theorem). The Darboux coordinates obtained by Arnold-Liouville's 
theorem are called \emph{action-angle coordinates}.

In \cite{W1}, \cite{W2} and \cite{W3}, Wolpert developed a very nice theory describing the symplectic geometry of $\FS$ in relation to
Fenchel-Nielsen coordinates. Let us present some of his results. In the following, $\FS$ is equipped with its standard symplectic structure
$\omega_G$ ($ = \omega_{WP}$ under the identification $\TS \approx \FS$, see section $\ref{fuchhol}$).

\begin{theorem}[Wolpert]\label{eree}
Let $\gamma$ be any nontrivial free homotopy class of simple closed curves on $S$. The flow of the Hamiltonian vector field $-X_{l_\gamma}$ is precisely
the twist flow $\mathrm{tw}_\gamma$.
\end{theorem}
In other words, \begin{equation}\frac{\partial}{\partial \tau_\gamma} = -X_{l_\gamma}~.\end{equation}

\begin{theorem}[Wolpert]\label{wolfor}
Let $\gamma$ and $\gamma'$ be distinct nontrivial free homotopy classes of simple closed curves on $S$. Then at any point $X\in\FS$,
\begin{equation}{\omega_G}\left(\frac{\partial}{\partial \tau_\gamma},\frac{\partial}{\partial \tau_{\gamma'}}\right) = \sum_{p\in (\gamma^X \cap {\gamma'}^X)}
\cos \theta_p~,\end{equation} where $\theta_p$ is the angle between the geodesics $\gamma^X$ and ${\gamma'}^X$ at $p$.
\end{theorem}

A direct consequence of these two theorems is:
\begin{theorem}
If $\alpha$ is a pants decomposition of $S$, then Fenchel-Nielsen length functions $l_{\alpha_i}$ define an integrable Hamiltonian system. The Hamiltonian
$\R^N$-action associated to this system is the twist flow $\mathrm{tw}_\gamma$.
\end{theorem}

Wolpert also shows that 
\begin{propo}[Wolpert]
If $\alpha$ is a pants decomposition of $S$, then for any $i,j \in \{1,\dots,N\}$
\begin{equation}{\omega_G}\left(\frac{\partial}{\partial l_{\alpha_i}},\frac{\partial}{\partial l_{\alpha_j}}\right) = 0~.\end{equation}
\end{propo}

It follows that we are in the best possible situation:
\begin{theorem}[Wolpert]\label{FNdarb}
Let $\alpha$ be a pants decomposition of $S$. Fenchel-Nielsen length and twist parameters associated to $\alpha$
are respectively action and angle variables for 
the integrable Hamiltonian system defined by the functions $l_{\alpha_i}$. In particular, Fenchel-Nielsen coordinates are Darboux
coordinates for the symplectic structure: \begin{equation}\omega_G = \sum_{i=1}^N dl_{\alpha_i} \wedge d\tau_{\alpha_i}~.\end{equation}
\end{theorem}
It is remarkable in particular that this does not depend on the choice of the pants decomposition $\alpha$.

\subsection{Complex Fenchel-Nielsen coordinates}

Kourouniotis (in \cite{K3}, see also \cite{K1} and \cite{K2}) and Tan (in \cite{tan}) introduced a system of global holomorphic coordinates 
$(l^\C,\tau^\C) : \QFS \rightarrow \C^{N} \times \C^N$ that can be thought of as a complexification of Fenchel-Nielsen coordinates on the Fuchsian slice $\FS$.
We outline this construction and refer to \cite{K3}, \cite{tan} and also \cite{series} for details.

\subsubsection*{Complex distance and displacement in hyperbolic space}\label{reret}

Let $\alpha$ and $\beta$ be two geodesics in the hyperbolic space $\HH^3$.
The \emph{complex distance} between $\alpha$ and $\beta$ is the complex number $\sigma = \sigma(\alpha,\beta)$ (defined modulo $2i\pi\Z$) 
such that $\mathrm{Re}(\sigma)$ is the hyperbolic distance between
$\alpha$ and $\beta$ and $\mathrm{Im}(\sigma)$ is the angle between them (meaning the angle between the two planes 
containing their common perpendicular and either $\alpha$ or $\beta$). In the upper half-space model $\HH^3 = \C \times \R^*_+$,
after applying an isometry so that $\alpha$ has endpoints $(u,-u)$ and $\beta$ has endpoints $(p,-p)$ (where $u, p \in \C\mathbf{P}^1$),
$\sigma$ is determined by $e^\sigma u = p$. Note that one has to be careful about orientations and sign to define $\sigma$ unambiguously, see
\cite{K3} and \cite{series} for details.

Let $f$ be a non-parabolic isometry of $\HH^3$ different from the identity, and $\beta$ a geodesic perpendicular to the axis of $f$.
The \emph{complex displacement} of $f$ is the complex distance $\varphi$ between $\beta$ and $f(\beta)$. If $f$ is represented by a matrix $A \in \SL_2(\C)$,
the complex displacement of $f$ is given by \begin{equation}2 \cosh \left(\frac{\varphi}{2}\right) = \mathrm{tr}(A)~.\end{equation}
The complex displacement and oriented axis of a non-parabolic
isometry determine it uniquely.

\subsubsection*{Right-angled hexagons and pair of pants in hyperbolic space}

An (oriented skew) \emph{right-angled} hexagon in $\HH^3$ is a cyclically ordered set of six oriented geodesics $\alpha_k$ indexed
by $k \in \Z/6\Z$, such that $\alpha_k$ intersects $\alpha_{k+1}$ orthogonally. Define the complex length of the ``side'' $\alpha_k$
by $\sigma_k=\sigma(\alpha_{k-1},\alpha_{k+1})$.

\begin{propo}The following relations are showed in \cite{fenchel}:\\
Sine rule:
\begin{equation}\frac{\sinh \sigma_1}{\sinh \sigma_4} = \frac{\sinh \sigma_3}{\sinh \sigma_6} = \frac{\sinh \sigma_5}{\sinh \sigma_2}\end{equation}
Cosine rule: \begin{equation}\cosh \sigma_n = \frac{\cosh \sigma_{n+3}-\cosh \sigma_{n+1}\cosh \sigma_{n-1}}{\sinh \sigma_{n+1}\sinh \sigma_{n-1}}\end{equation}
\end{propo}

Using these formulas, one shows that assigning complex lengths on every other side determines a unique right-angled hexagon in $\HH^3$ up to
(possibly orientation-reversing) isometry. In \cite{K3} and \cite{tan}, it is showed the the construction of a hyperbolic pair of pants 
by gluing two right-angled hexagons can be extended to $\HH^3$. Such a pair of pants is thus uniquely determined by the complex lengths of its boundary
components. In terms of holonomy (\cite{K3}):
\begin{propo}
Let $P$ be a topological pair of pants and $\sigma_1$, $\sigma_2$, $\sigma_3 \in \C_+$ ({\it i.e}. with $\mathrm{Re}(\sigma_i) > 0)$. 
There is a unique representation up to conjugation $$\rho : \pi_1(P) = \langle c_1,c_2,c_3 ~|~ c_1c_2c_3 = 1 \rangle \to \PSL_2(\C)$$
such that $\mathrm{tr}(\rho(c_i)) = -2 \cosh \sigma_i$.
\end{propo}

\subsubsection*{Complex lengths and complex twisting in the quasi-Fuchsian space}

Let $Z \in \QFS$ be a quasi-Fuchsian structure on $S$ and $\rho : \pi_1(S) \to \PSL_2(\C)$ its holonomy representation. For any nontrivial
free homotopy class of simple closed curves $\gamma$, define the \emph{complex length} of $\gamma$ as the complex displacement of the
hyperbolic isometry $\rho(\gamma)$. This defines a holomorphic function $l_\gamma^\C : \QFS \to \C_+$ \label{complendef}.
In the quasi-Fuchsian $3$-manifold $M$, $\rho(\gamma)$ corresponds to a geodesic of complex length
$l^\C_{\gamma}$, {\it i.e.} of hyperbolic length $\mathrm{Re}(l^\C_{\gamma})$ and torsion $\mathrm{Im}(l^\C_{\gamma})$.
It is easy to see that if $Z$ is a Fuchsian
structure, then $l_\gamma^\C(Z) = l_\gamma(Z)$. If $\alpha = (\alpha_1, \dots, \alpha_N)$ is a pants decomposition of $S$, we call
\begin{equation}l_\alpha^\C = (l_{\alpha_1}^\C, \dots, l_{\alpha_N}^\C) : \QFS \to (\C_+)^N\end{equation} the 
\emph{complex Fenchel-Nielsen length parameters}.

As a consequence of the previous discussion, if the complex lengths $l_{\alpha_1}^\C, \dots, l_{\alpha_N}^\C$ are fixed, a quasi-Fuchsian
structure on $S$ is determined by how the gluings of pair of pants occur along their common boundaries. Analogously to the Fuchsian case, this is prescribed 
by a complex parameter $\tau_{\alpha_i}^\C$, that we will call a complex twist parameter \label{comptwistdef}, that describes both the amount of twisting 
(by $\mathrm{Re}(\tau_{\alpha_i}^\C)$) and the amount of bending (by $\mathrm{Im}(\tau_{\alpha_i}^\C)$)
before gluing. $\tau_{\alpha_i}^\C$ can be more or less well defined as the complex distance between two adequate geodesics in $\HH^3$, but the definition is clearer
in terms of the effect of complex-twisting by $\tau_{\alpha_i}^\C$ on the holonomy of the glued pairs of pants (see \cite{K3}, \cite{Gsl2}).

As in the Fuchsian case, it is the complex twist flow $\mathrm{tw}_\gamma^\C$ along a simple closed curve $\gamma$ that is well-defined
rather than the twist parameter $\tau_{\alpha_i}^\C$ itself, although the complex twist vector field $\dfrac{\partial}{\partial \tau_\gamma^\C}$ is well-defined.
Let us mention this flow is called \emph{bending} by Kourouniotis and corresponds to
(or is a generalization of) what other authors have called \emph{quakebends} or \emph{complex earthquakes} discussed by Epstein-Marden \cite{EP}, 
Goldman \cite{Gsl2}, McMullen \cite{mcmullenearthquakes}, Series \cite{series} among others. It is not hard to see that starting from a Fuchsian structure $Z$,
complex twisting by $t = t_1+i t_2\in\C$ is described as the composition of twisting by $t_1$ on $\FS$ and then \emph{projective grafting}
 by $t_2$ (see e.g. \cite{dumas}
for a presentation of projective grafting).

Choosing a holomorphic section to $l_\alpha^\C$ determines \emph{complex twist coordinates}
$\tau_{\alpha}^\C = (\tau_{\alpha_1}^\C, \dots, \tau_{\alpha_N}^\C) : \QFS \to \C^N$. We will call $(l_{\alpha}^\C,\tau_{\alpha}^\C)$ complex Fenchel-Nielsen
coordinates. The conclusion of our discussion is the theorem:
\begin{theorem}[Kourouniotis, Tan]
Complex Fenchel-Nielsen coordinates $(l_{\alpha}^\C,\tau_{\alpha}^\C)$ are global holomorphic coordinates on $\QFS$. 
They restrict to the classical
Fenchel-Nielsen coordinates $(l_{\alpha},\tau_{\alpha})$ on the Fuchsian slice $\FS$.
\end{theorem}

\subsection{Platis' symplectic structure}

In \cite{platis}, Platis develops a complex version of Wolpert's theory on the quasi-Fuchsian space, here are some of his results.

First there is a complex version of Wolpert's formula \ref{wolfor}:
\begin{theorem}
There exists a complex symplectic structure $\omega_P$ on $\QFS$ such that if $\gamma$ and $\gamma'$ are distinct nontrivial free homotopy classes of simple closed 
curves on $S$, then at any point $Z\in\QFS$ with holonomy $\rho$
\begin{equation}{\omega_P}\left(\frac{\partial}{\partial \tau_\gamma^\C},\frac{\partial}{\partial \tau_{\gamma'}^\C}\right) = \sum_{p\in (\gamma \cap {\gamma'})}
\cosh \sigma_p~,\end{equation} where $\sigma_p$ is the complex distance between the geodesics $\rho(\gamma)$ and $\rho(\gamma')$.
\end{theorem}

He also shows the complex analogous of theorem \ref{eree} in the complex symplectic manifold $(\QFS,\omega_P)$:
\begin{theorem}
Let $\gamma$ be any nontrivial free homotopy class of simple closed curves on $S$. The complex flow of the Hamiltonian vector field $-X_{l_\gamma^\C}$ is precisely
the complex twist flow $\mathrm{tw}_\gamma^\C$.
\end{theorem}

As in the Fuchsian case, it follows from these two theorems that complex Fenchel-Nielsen length functions associated to a pants decomposition define
a complex Hamiltonian integrable system. Furthermore, he proves that the striking theorem \ref{FNdarb} is still true in its complex version on $(\QFS,\omega_P)$:
\begin{theorem}
If $\alpha$ is any pants decomposition of $S$, complex Fenchel-Nielsen coordinates are Darboux
coordinates for the complex symplectic structure $\omega_P$: \begin{equation}\omega_P = \sum_{i=1}^N dl_{\alpha_i}^\C \wedge d\tau_{\alpha_i}^\C~.\end{equation}
\end{theorem}

\section{Comparing symplectic structures}\label{compsst}

\subsection{Analytic continuation}\label{ancont}

We are going to show the following proposition, which implies that two complex symplectic structures agree on $\CPS$ if and only if they agree
in restriction to tangent vectors to the Fuchsian slice $\FS$:

\begin{propo}\label{puit}
 Let $\omega$ be a closed $(2,0)$-form on $\CPS$ and $\sigma_{\cal F} : \TS \to \CPS$ be the Fuchsian section (as in (\ref{fufuch})).
If ${\sigma_{\cal F}}^*\omega$ vanishes
on $\TS$, then $\omega$ vanishes on $\CPS$.
\end{propo}

The proof of this proposition is based on analytic continuation. In order to use this argument, we recall a few definitions and show some elementary facts
regarding totally real submanifolds of complex manifolds.

\begin{defin}
Let $M$ be a complex manifold and $N \subset M$ be a real submanifold. $N$ is called \emph{totally real} if the following holds:
\begin{equation}\forall x\in N,~~ T_xN \cap JT_xN = \{0\}~,\end{equation}
where $J$ is the almost complex structure on $M$.  \label{almcs}
\end{defin}

If moreover, $N$ has maximal dimension $\dim_\R N = \dim_\C M$, we say that $N$ is a \emph{maximal totally real submanifold} of $M$. There are
several characterizations of maximal totally real analytic submanifolds, seemingly stronger than the definition, as in the following:

\begin{propo}\label{atrs}
Let $M$ be a complex manifold of dimension $n$ and $N \subset M$ be a real submanifold. The following are equivalent:
\begin{itemize}
 \item{(i)} $N$ is a maximal totally real analytic submanifold of $M$.
 \item{(ii)} $N \subset M$ locally looks like $\R^n \subset \C^n$. More precisely: for any $x\in N$, there is a holomorphic chart $z : U \rightarrow V$ 
 where $U$ is an open set in $M$ containing $x$ and $V$ is an open set in $\C^n$, such that $z(U\cap N) = V \cap \R^n$.
 \item{(iii)} There is an antiholomorphic involution $\chi : M' \rightarrow M'$ where $M'$ is a neighborhood of $N$ in $M$, such that $N$ is the set 
 of fixed points of $\chi$.
\end{itemize}
\end{propo}
If $N$ satisfies one (equivalently all) of these conditions, $M$ is said to be a \emph{complexification} of $N$. Let us mention that any
real-analytic manifold can be complexified.

\begin{proof}
It is fairly easy to see that both $(ii)$ and $(iii)$ imply $(i)$, and that in fact $(ii)$ and $(iii)$ are equivalent. Let us show that $(i)$ implies $(ii)$.
Using holomorphic charts, it is clearly enough to prove this in the case where $N$ is a maximal totally real analytic submanifold of $\C^n$. Let $m\in N \subset \C^n$, 
there is a real-analytic parametrization $\varphi : D \to N$, where $D$ is a small open disk centered at the origin in $\R^n$, such that $\varphi(0) = m$ and 
$d \varphi(0) \neq 0$. The map $\varphi$ is given by a convergent power series $\varphi(x) = \sum_{|\alpha| = n} a_\alpha x^\alpha$ for all $x\in D$, 
where the sum is taken over all multi-indices $\alpha$ of length $n$, and the $a_\alpha$ are coefficients in $\C^n$. In order to extend $\varphi$ to a holomorphic map
$\Phi : D' \to M$ where $D'$ is the disk in $\C^n$ such that $D = D'\cap \R^n$, we can just replace $x\in D$ by $z \in D'$ in the expression of $\varphi$:
define $\Phi(z) = \sum_{|\alpha| = n} a_\alpha z^\alpha$. This power series converges in $D'$ because it has the same radius of convergence as its real counterpart.
Moreover, if $D'$ is small enough, $\Phi$ is a biholomorphism onto its image because $d\Phi(0) = d\varphi(0) \neq 0$. This shows that $(i)$ implies $(ii)$ (just
take the chart given by $\Phi^{-1}$)\footnote{Note
that the simplicity of this proofs relies on a little trick: the actual complexification of $\varphi$ is a map $\varphi : D'
\rightarrow \C^{2n}$ (and not $\C^n$), where $a_\alpha$ is seen as a real vector in $\C^n$.}.
\end{proof}

Keeping in mind that we want to consider the Fuchsian slice in $\CPS$, we make this last general observation on totally real submanifolds:

\begin{propo}
Let $V$ be a complex manifold. The diagonal $\Delta$ in $V \times \overline{V}$\footnote{$\overline{V}$ denotes the manifold
$V$ equipped with the opposite complex structure.} is a maximal totally real analytic submanifold.
\end{propo}

\begin{proof}
This is a direct consequence of characterization $(iii)$ in the previous proposition: just take the antiholomorphic involution $\chi : V \times \bar{V} \to
V \times \bar{V}$ defined by $\chi(x,y) = (y,x)$.
\end{proof}

An immediate application of this is that the Fuchsian slice $\FS$ is a maximal totally real analytic submanifold of $\CPS$ (as was already
pointed out in section \ref{fuchhol}): it is the image of the diagonal of $\TS \times {\cal T}(\overline{S})$\footnote{
Recall that ${\cal T}(\overline{S})$ is canonically identified with $\overline{\TS}$.} by the holomorphic embedding $\beta^+$
 (see section \ref{fqf}).
Another way to see this is that the quasi-Fuchsian space $\QFS = \mathrm{Im}(\beta^+) \subset \CPS$ is equipped with a canonical antiholomorphic involution,
which justs consists in ``turning a quasi-Fuchsian $3$-manifold upside down'', and $\FS$ is the set of fixed points of this involution.

Now, we prove this first elementary analytic continuation theorem:

\begin{propo}
Let $M$ be a connected complex manifold and $N \subset M$ be a maximal totally real submanifold. If $f : M \to \C$ is a holomorphic function that
vanishes on $N$, then $f$ vanishes on $M$.
\end{propo}

\begin{proof}
By the identity theorem for holomorphic functions, it is enough to show that $f$ vanishes on a small open neighborhood $U$ of some point $x\in N$.
If $N$ is analytic, this is an straightforward consequence of characterization $(ii)$ in Proposition \ref{atrs}. Let us produce a proof that does not assume
analyticity of $N$. Since the restriction $f_{|N}$ vanishes identically, we have ${(df)}_{|TN} = 0$. Using the fact that $T_x M = T_xN \oplus J T_x N$
for all $x\in N$ and the holomorphicity of $f$, it is easy to derive that $df$ vanishes at all points of $N$. In particular, if 
$z=(z_k)_{1\leqslant k \leqslant n} : U \to \C^n$ is
a holomorphic chart, the partial derivatives $\dfrac{\partial f}{\partial z_k}$ vanish on $N$. But those are again holomorphic functions, so we can use
the same argument: their partial derivatives must vanish on $N$. By an obvious induction, we see that all partial derivatives of $f$ (at any order)
vanish at points of $N$. Since $f$ is holomorphic, this implies that $f=0$.
\end{proof}

We can now finally prove:
\begin{propo}
Let $M$ be a connected complex manifold and $\sigma : N \rightarrow M$ be a maximal totally real embedding. If $\omega$ is a closed $(2,0)$-form 
on $M$ such that $\sigma^* \omega = 0$, then $\omega = 0$.
\end{propo}

\begin{proof}
We can assume that $N \subset M$, so the hypothesis is that $\omega_{|TN} = 0$. Since $T_xM = T_xN \oplus JT_xN$ for any $x\in N$ and $\omega$ is of type $(2,0)$,
it is easy to see that $\omega$ vanishes at points of $N$. Now, recall that a closed $(2,0)$-form  is holomorphic.
Let $z=(z_k)_{1\leqslant k \leqslant n} : U \to \C^n$ be a holomorphic chart in a neighborhood of a point $x\in N$, $\omega$ has an expression
of the form $\omega = \sum_{j,k} f_{jk} dz_j \wedge dz_k$ where $f_{jk}$ are holomorphic
functions on $U$. Since $\omega$ vanishes at points of $N$, the functions $f_{jk}$ vanish on $U \cap N$, and we derive from the previous proposition that they
actually vanish on $U$. We thus have $\omega_{|U} = 0$, and it follows once again from the identity theorem (taken in charts) that $\omega$ vanishes on $M$.
\end{proof}

An immediate consequence of this, together with (\ref{wgfsss}), is that a complex symplectic structure on $\CPS$ agrees with the standard complex
symplectic structure if and only if it induces the Weil-Petersson K\"ahler form on the Fuchsian slice:

\begin{theorem}\label{bilibili}
Let $\omega$ be a complex symplectic structure on $\CPS$. Then $\omega = \omega_G$ if and only if ${(\sigma_{\cal F})}^* \omega = \omega_{WP}$.
\end{theorem}

\begin{proof}
By the previous proposition, $\omega = \omega_G$ if and only if $(\sigma_{\cal F})^*(\omega-\omega_G) = 0$. But by (\ref{wgfsss}),
$\sigma_{\cal F}^*\omega_G = \omega_{WP}$, therefore $(\sigma_{\cal F})^*(\omega-\omega_G) = (\sigma_{\cal F})^*\omega - \omega_{WP}$.
\end{proof}

\subsection{The affine cotangent symplectic structures}

Recall (see section $\ref{affid}$) that any section $\sigma :\TS \to \CPS$ determines an affine identification $\tau^{\sigma} : \CPS \stackrel{\sim}{\to} T^*\TS$ and thus
a complex-valued non-degenerate $2$-form $\omega^{\sigma} = ({\tau^{\sigma}})^*\omega$ on $\CPS$. $\omega^{\sigma}$ is a complex symplectic
structure on $\CPS$ if and only if $\sigma$ is a holomorphic section to $p$. We will now answer the question: for which holomorphic sections $\sigma$
does $\omega^{\sigma}$ agree with the standard symplectic structure $\omega_G$?

As a direct consequence of theorem \ref{bilibili}, together with Proposition \ref{craca}, we show:

\begin{theorem}\label{malauvent}
Let $\sigma : \TS \rightarrow \CPS$ be a section to $p$. Then $\omega^\sigma$ agrees with the standard complex symplectic structure $\omega_G$ on $\CPS$
if and only if $\sigma_{\cal F} - \sigma$ is a primitive for the Weil-Petersson metric on $\TS$:
\begin{equation}\omega^{\sigma} = \omega_G ~ \Leftrightarrow ~ d(\sigma_{\cal F} - \sigma) = \omega_{WP}~.\end{equation}
More generally, if $c$ is some complex constant,
\begin{equation}\omega^{\sigma} = c \omega_G ~ \Leftrightarrow ~ d(\sigma_{\cal F} - \sigma) = c \omega_{WP}~.\end{equation}
\end{theorem}

\begin{proof}
By Theorem \ref{bilibili}, $\omega^\sigma = \omega_G$ if and only if ${(\sigma_{\cal F})}^* \omega^\sigma = \omega_{WP}$.
However, it follows from Proposition \ref{craca} that
\begin{eqnarray*}
{(\sigma_{\cal F})}^* \omega^\sigma &=& {(\sigma_{\cal F})}^*\left[\omega^\sigma_{\cal F} -p^*d(\sigma - \sigma_{\cal F})\right]\\
&=& {(\sigma_{\cal F})}^*((\tau^{\sigma_{\cal F}})^* \omega_{\mathrm{can}}) - {(\sigma_{\cal F})}^*(p^*d(\sigma - \sigma_{\cal F}))\\
&=& {(\tau^{\sigma_{\cal F}} \circ \sigma_{\cal F})}^* \omega_{\mathrm{can}} - {(p \circ \sigma)}^*d(\sigma - \sigma_{\cal F})\\
&=& {s_0}^* \omega_{\mathrm{can}} - \mathrm{id}^*d(\sigma - \sigma_{\cal F})\\
&=& -d(\sigma - \sigma_{\cal F})~.
\end{eqnarray*}
Let us make a couple of comments on this calculation: recall that $\omega_{\mathrm{can}}$ denotes the canonical symplectic 
structure on $T^*\TS$; $\tau^{\sigma} \circ \sigma
= s_0$ is the characterization of $\tau^\sigma$; and ${s_0}^*\omega = 0$ because the zero section $s_0$ is a Lagrangian in $T^*\TS$ (see 
section \ref{compsympcot}). Of course, the proof of the apparently more general second statement is just the same. 
\end{proof}

Now, McMullen proved in \cite{mcmullenkahler} the following theorem:
\begin{theorem}[McMullen \cite{mcmullenkahler}]\label{mcm}
If $\sigma$ is any Bers section, then \begin{equation}d(\sigma_{\cal F} - \sigma) = -i\omega_{WP}~.\end{equation}
\end{theorem}

Using this result, we eventually obtain as a corollary of Theorem \ref{malauvent}:
\begin{theorem}\label{main1}
If $\sigma : \TS \rightarrow \CPS$ is any Bers section, then \begin{equation}\omega^\sigma = -i \omega_G~.\end{equation}
\end{theorem}

In particular, we can deduce from this identification the following properties:
\begin{coro}\label{tropo}
Consider the space $\CPS$ equipped with its standard symplectic structure $\omega_G$. Then
\begin{enumerate}
 \item The canonical projection $p : \CPS \to \TS$ is a Lagrangian fibration.
 \item Bers slices are the leaves of a Lagrangian foliation of the quasi-Fuchsian space $\QFS$.
\end{enumerate}
\end{coro}

We also derive an explicit expression of $\omega_G(u,v)$ when $u$ is a vertical tangent vector (by Proposition \ref{vertg}):
\begin{coro}\label{ikik}
Let $u$, $v$ be tangent vectors at $Z \in \CPS$ such that $u$ is vertical, {\it i.e.}
$p_* u= 0$. Then  \begin{equation}\omega_G(u,v) = i \langle u,p_*v \rangle~.\end{equation}
\end{coro}

Looking back at Proposition \ref{craca}, we also get the expression of the $2$-form $\omega^{\sigma_{\cal F}}$
obtained under the Fuchsian identification:

\begin{coro}\label{fghh}
Let $\sigma_{\cal F} : \TS \to \CPS$ be the Fuchsian section. Then
\begin{equation}\omega^{\sigma_{\cal F}} = -i(\omega_G - p^* \omega_{WP})~.\end{equation}
\end{coro}

It should not come as a surprise that we see from this equality that $\omega^{\sigma_{\cal F}}$ vanishes on the Fuchsian slice.
Notice the equality between real symplectic structures: 
\begin{equation}\mathrm{Re}(\omega^{\sigma_{\cal F}}) = \mathrm{Im}(\omega_G) \label{mpmp}\end{equation}
and that $\mathrm{Re}(\omega^{\sigma_{\cal F}})$ is (half) the real canonical symplectic structure on $T^*\TS$ pulled back by the Fuchsian identification.

Finally, we note that McMullen's \emph{quasi-Fuchsian reciprocity} theorem showed in \cite{mcmullenkahler}
can be seen as a consequence of Theorem \ref{main1}. We will give a precise statement and proof of a generalized version of this 
theorem in the setting of convex cocompact $3$-manifolds (Theorem \ref{qfrcp}).

\subsubsection*{Generalizations in the setting of convex cocompact hyperbolic $3$-manifolds}

Let $\hat{M}$ be a compact $3$-manifold as in section \ref{cphyp}. We will use here the same notations as in \ref{cphyp}. Recall that we have defined
there a canonical holomorphic section $\beta : {\cal T}(\partial \hat{M}) \to {\cal CP}(\partial \hat{M})$.

McMullen and Takhtajan-Teo gave generalized versions of quasi-Fuchsian reciprocity, which they called \emph{Kleinian reciprocity}, notably in 
\cite{mcmullenkahler} (Appendix) and \cite{takteo}. In particular, {\bf Theorem 6.3} in \cite{takteo} says the following:

\newtheorem*{papaga}{Theorem}
\begin{papaga}
Let $\sigma_{\cal F} : {\cal T}(\partial \hat{M}) \to {\cal CP}(\partial \hat{M})$ denote the Fuchsian section. Then
$$d(\sigma_{\cal F} - \beta) = -i \omega_G~.$$
\end{papaga}

Since our theorem \ref{malauvent} above does not assume that $S$ is connected, we obtain:
\begin{theorem}\label{puuff}
Let $\omega_G$ be the standard complex symplectic structure on ${\cal CP}(\partial \hat{M})$ and $\omega^\beta = {(\tau^\beta)}^* \omega$ 
be the complex symplectic structure obtained
by the identification $\tau^\beta : {\cal CP}(\partial \hat{M}) \stackrel{\sim}{\to} T^*{\cal T}(\partial \hat{M})$ as in section \ref{affid}.
Then
\begin{equation}\omega^\beta = -i\omega_G~. \label{ratatac}\end{equation}
\end{theorem}

A first immediate corollary is that we recover Theorem \ref{lagemb}: $\beta$ is a Lagrangian embedding.

Another consequence of this theorem and of the fact that
the projections $p_k : {\cal CP}(S_k) \to {\cal T}(S_k)$ are Lagrangian (Theorem \ref{tropo}) is a generalization of Theorem \ref{main1}:
\begin{theorem}
Let $\sigma : \TS \to \CPS$ be a generalized Bers section (see section \ref{cphyp}). 
Then \begin{equation}\omega^\sigma = -i \omega_G~.\end{equation} \label{main2}
\end{theorem}
\begin{proof}
By definition, $\sigma$ is map defined by $\sigma = f_{(X_i),j}$ as in section \ref{cphyp}, where $S = S_j$ and $X_i$ is a fixed point in ${\cal T}(S_i)$
for $i \neq j$. Recall that $$\omega_G = (\pr_1)^* {\omega_G}^{(1)} + \dots + (\pr_N)^* {\omega_G}^{(N)}$$ where ${\omega_G}^{(k)}$ is the standard
complex symplectic structure on ${\cal CP}(S_k)$, and similarly $$\omega = (\pr_1)^* {\omega}^{(1)} + \dots + (\pr_N)^* {\omega}^{(N)}\footnote{
Be wary that in this equality, $\pr_k$ stands for the $k^{\textrm{th}}$ projection map $T^*{\cal T}(\partial \hat{M}) \rightarrow
T^*{\cal T}(S_k)$ (whereas it stood for the $k^{\textrm{th}}$ projection map ${\cal CP}(\partial \hat{M}) \rightarrow
{\cal CP}(S_k)$ in the previous equality). We apologize for these (slightly) misleading notations.}$$
where ${\omega}^{(k)}$ is the canonical complex symplectic structure on $T^*{\cal T}(S_k)$. The equality (\ref{ratatac}) can thus be rewritten:
$${(\pr_1 \circ \tau^\beta)}^* {\omega}^{(1)} + \dots + {(\pr_N \circ \tau^\beta)}^* {\omega}^{(N)} = 
-i\left[(\pr_1)^* {\omega_G}^{(1)} + \dots + (\pr_N)^* {\omega_G}^{(N)}\right]~.$$
Fix $Z_i \in P(X_i)$ for $i \neq j$ and let us pull back this equality on ${\cal CP}(S_j)$ by the map
$$\tilde{\iota_{(Z_i)}} : 
\begin{array}{ccl}
{\cal CP}(S_j) & \rightarrow & {\cal CP}(\partial \hat{M})\\
  Z & \mapsto & (Z_1, \dots, Z_{j-1},Z,Z_{j+1}, \dots, X_N)~.
\end{array}$$
For $k \neq j$, the map $\pr_k \circ \tau^\beta \circ \tilde{\iota}_{(Z_i)}$ takes values in the fiber ${T_{X_k}}^*{\cal T}(S_k)$, so that
${({\tilde{\iota}}_{(Z_i)})}^*\left({(\pr_k \circ \tau^\beta)}^* {\omega}^{(k)}\right) = 0$.
Similarly, the map $\pr_k \circ \tilde{\iota}_{(Z_i)}$ maps into the fiber $P(X_k) \subset {\cal CP}(S_k)$, so that
${({\tilde{\iota}}_{(Z_i)})}^*\left((\pr_k)^* {\omega_G}^{(k)}\right) = 0$ because $p_k$ is a Lagrangian fibration.
For $k=j$, $\pr_k \circ \tau^\beta \circ \tilde{\iota}_{(Z_i)}$ is the map $\tau^\sigma : {\cal CP}(S_j) \to T^*{\cal T}(S_j)$ and 
$\pr_k \circ \tilde{\iota}_{(Z_i)}$ is the identity in ${\cal CP}(S_j)$. We therefore obtain the desired equality 
$(\tau^\sigma)^* \omega^{(j)} = -i {\omega_G}^{(j)}$.
\end{proof}

An immediate corollary of this is:
\begin{coro}
Generalized Bers sections $\TS \to \CPS$ are Lagrangian embeddings.
\end{coro}

Another corollary of Theorem \ref{main2} and Theorem \ref{malauvent} is a generalization of McMullen's Theorem \ref{mcm}:
\begin{coro}\label{coui}
Let $\sigma : \TS \to \CPS$ be a generalized Bers section. Then \begin{equation}d(\sigma_{\cal F} - \sigma) = -i\omega_{WP}~.\end{equation}
\end{coro}

Finally, we show a generalized version of McMullen's ``quasi-Fuchsian reciprocity''.
To this end, we introduce the notion of ``reciprocal generalized Bers embeddings'': with the notations of section \ref{cphyp}, 
we say that $f : {\cal T}(S_j) \to {\cal CP}(S_k)$
and $g : {\cal T}(S_k) \to {\cal CP}(S_j)$ are \emph{reciprocal generalized Bers embeddings} if $f = f_{(X_i)_{i \neq j}, k}$ and $g = f_{(X_i)_{i \neq k}, j}$
for some fixed $X = (X_1, \dots, X_N) \in  {\cal T}(\partial \hat{M})$. Since $f$ and $g$ take values in affine spaces, one can consider 
their derivatives:
\begin{eqnarray}
D_{X_j}f : T_{X_j}{\cal T}(S_j) \to T_{X_k}^*{\cal T}(S_k)\\
D_{X_k}g : T_{X_k}{\cal T}(S_k) \to T_{X_j}^*{\cal T}(S_j)
\end{eqnarray}
\begin{theorem}\label{qfrcp}
Let $f : {\cal T}(S_j) \to {\cal CP}(S_k)$
and $g : {\cal T}(S_k) \to {\cal CP}(S_j)$ be reciprocal generalized Bers embeddings as above. Then $D_{X_j}f$ and $D_{X_k}g$ are dual maps.
In other words, for any $\mu\in T_{X_j}{\cal T}(S_j)$ and $\nu\in T_{X_k}{\cal T}(S_k)$,
\begin{equation} \langle D_{X_j}f(\mu),\nu \rangle = \langle \mu,D_{X_k}g(\nu) \rangle~.\end{equation}
\end{theorem}
\begin{proof}
The fact that $\beta : {\cal T}(\partial \hat{M}) \to {\cal CP}(\partial \hat{M})$ is Lagrangian is written
\begin{equation*}\beta^* \omega_G = {(\pr_1 \circ \beta)}^* {\omega_G}^{(1)} + \dots + {(\pr_N \circ \beta)}^* {\omega_G}^{(N)} = 0~,\end{equation*}
where ${\omega_G}^{(i)}$ is the standard
complex symplectic structure on the component ${\cal CP}(S_i)$. Let $\mu\in T_{X_j}{\cal T}(S_j)$ and $\nu\in T_{X_k}{\cal T}(S_k)$, and define 
$U,V \in T_X{\cal T}(\partial \hat{M})$ by 
$$U_i = \left\{\begin{array}{cl}
         0 & \textrm{ for } i \neq j\\
         \mu & \textrm{ for } i = j
        \end{array}\right.$$
and
$$V_i = \left\{\begin{array}{cl}
         0 & \textrm{ for } i \neq k\\
         \nu & \textrm{ for } i = k
        \end{array}\right.~.$$
Note that ${(\pr_i \circ \beta)}_* U$ is vertical in ${\cal CP}(S_i)$ whenever $i\neq j$ 
(resp.  ${(\pr_i \circ \beta)}_* V$ is vertical in ${\cal CP}(S_i)$ whenever $i\neq k$). 
Since the fibers of ${\cal CP}(S_i)$ are isotropic (see Theorem \ref{tropo}),
it follows that $\left({(\pr_i \circ \beta)}^* {\omega_G}^{(i)}\right)(U,V) = 0$ whenever $i \not\in \{j,k\}$. For $i=k$, ${(\pr_i \circ \beta)}_* U = D_{X_j}f(\mu)$
is still vertical and ${(p_k)}_*\left({(\pr_i \circ \beta)}_*V\right) = \nu$ so we can derive from Theorem \ref{ikik} that 
$\left({(\pr_k \circ \beta)}^* {\omega_G}^{(k)}\right)(U,V) = i \langle D_{X_j}f(\mu), \nu \rangle $. 
Similarly, for $i=j$ we have
$$\left({(\pr_j \circ \beta)}^* {\omega_G}^{(j)}\right)(U,V) = - \left({(\pr_j \circ \beta)}^* {\omega_G}^{(j)}\right)(V,U)
= -i \langle D_{X_k}g(\nu), \mu \rangle~.$$
In the end, $0 = \left(\beta^*\omega_G\right)(U,V) = i \langle D_{X_j}f(\mu), \nu \rangle -i \langle D_{X_k}g(\nu), \mu \rangle$ as desired.
\end{proof}

Note that Corollary \ref{coui} is not an immediate consequence of ``Kleinian reciprocity'': the proof requires Lagrangian information.
On the other hand, Steven Kerckhoff pointed out to me that Theorem \ref{qfrcp} can be derived from Kleinian reciprocity 
without using our previous results, 
rightly so
(the proof is easily adapted taking $\omega^\beta$ instead of $\omega_G$, avoiding the use of the symplectic structure on $\CPS$).

\subsection{Darboux coordinates}

It is an immediate consequence of the analytical continuation property \ref{bilibili} and Wolpert's theorem \ref{FNdarb} that
complex Fenchel-Nielsen coordinates are Darboux coordinates for the standard symplectic structure on the quasi-Fuchsian space:

\begin{theorem}\label{main3}
Let $\alpha$ be any pants decomposition of $S$. Complex Fenchel-Nielsen coordinates $(l_\alpha^\C,\beta_\alpha^\C)$ on the 
quasi-Fuchsian space $\QFS$ are
Darboux coordinates for the standard complex symplectic structure:
\begin{equation}\omega_G = \sum_{i=1}^N dl_{\alpha_i}^\C \wedge d\tau_{\alpha_i}^\C~.\end{equation}
\end{theorem}
\begin{proof}
Of course, Theorem \ref{bilibili} is still true when replacing $\CPS$ by any connected neighborhood of the Fuchsian slice $\FS$, such as the 
quasi-Fuchsian space $\QFS$.
Let $\omega =  \sum_{i=1}^N dl_{\alpha_i}^\C \wedge d\tau_{\alpha_i}^\C$. Since $l_{\alpha_i}^\C$ and $\tau_{\alpha_i}^\C$ are holomorphic, $\omega$ is 
a complex symplectic structure on $\QFS$. In restriction to the Fuchsian slice, $\iota^* \omega = \sum_{i=1}^N dl_{\alpha_i} \wedge d\tau_{\alpha_i}$
where $(l_\alpha,\tau_\alpha)$ are the classical Fenchel-Nielsen coordinates. By Wolpert's Theorem \ref{FNdarb}, it follows that ${(\sigma^{\cal F})}^* \omega
= \omega_{WP}$. This proves that $\omega = \omega_G$ according to Theorem \ref{bilibili}.
\end{proof}

Of course, this shows in particular:
\begin{coro}\label{za}
Platis' symplectic structure $\omega_P$ is equal to the standard complex symplectic structure $\omega_G$ in restriction to the 
quasi-Fuchsian space $\QFS$.
\end{coro}
Notice how the analytic continuation argument provides a very simple alternative proof of Platis' result that the symplectic structure 
$\sum_{i=1}^N dl_{\alpha_i}^\C \wedge d\tau_{\alpha_i}^\C$ does not depend on a choice of a pants decomposition (relying,
of course, on Wolpert's result).

In \cite{Gsl2}, Goldman gives a fairly extensive description of the complex symplectic structure $\omega_G$ on the character variety
${\cal X}(S,\SL_2(\C))$, discussing in particular the ``Hamiltonian picture''. We recover that the Hamiltonian flow of a complex
length function is the associated complex twist flow.

\bibliographystyle{alpha}
\bibliography{Article1}

\end{document}